\documentclass[leqno,english]{amsart}
\usepackage{ae,aecompl}
\usepackage{helvet}

\usepackage[T1]{fontenc}
\usepackage[latin9]{inputenc}
\usepackage{color}
\usepackage{babel}
\usepackage{enumitem}
\usepackage{amstext}
\usepackage{amsthm}
\usepackage{amssymb}
\usepackage{microtype}
\usepackage[unicode=true,pdfusetitle,
 bookmarks=true,bookmarksnumbered=true,bookmarksopen=false,
 breaklinks=false,pdfborder={0 0 0},pdfborderstyle={},backref=page,colorlinks=true]
 {hyperref}
\hypersetup{
 citecolor=blue,linkcolor=red}

\makeatletter

\providecommand{\tabularnewline}{\\}

\numberwithin{equation}{section}
\numberwithin{figure}{section}
\numberwithin{table}{section}
\theoremstyle{plain}
\newtheorem{thm}{\protect\theoremname}[section]
\theoremstyle{definition}
\newtheorem{defn}[thm]{\protect\definitionname}
\theoremstyle{remark}
\newtheorem{notation}[thm]{\protect\notationname}
\theoremstyle{definition}
\newtheorem{example}[thm]{\protect\examplename}
\theoremstyle{remark}
\newtheorem{rem}[thm]{\protect\remarkname}
\theoremstyle{plain}
\newtheorem{prop}[thm]{\protect\propositionname}
\theoremstyle{plain}
\newtheorem{lem}[thm]{\protect\lemmaname}
\theoremstyle{plain}
\newtheorem{cor}[thm]{\protect\corollaryname}

\usepackage{tikz,fancyhdr,enumitem}
\usetikzlibrary{scopes,calc,intersections,through,backgrounds,arrows.meta,decorations.pathmorphing,positioning,fit,petri,decorations.markings,matrix}
\usepackage{hyperref,bookmark,dsfont,nicefrac,bbm}
\usepackage[all]{xy}
\hypersetup{linktocpage}

\usepackage{amsmath}
\usepackage{amssymb}
\usepackage{graphicx}
\usepackage{calc}

\newcommand{\op}{
  \mathop{
    \vphantom{\bigoplus} 
    \mathchoice
      {\vcenter{\hbox{\resizebox{\widthof{$\displaystyle\bigoplus$}}{!}{$\boxplus$}}}}
      {\vcenter{\hbox{\resizebox{\widthof{$\bigoplus$}}{!}{$\boxplus$}}}}
      {\vcenter{\hbox{\resizebox{\widthof{$\scriptstyle\oplus$}}{!}{$\boxplus$}}}}
      {\vcenter{\hbox{\resizebox{\widthof{$\scriptscriptstyle\oplus$}}{!}{$\boxplus$}}}}
  }\displaylimits 
}

\renewcommand\ell{l}

\subjclass[2020]{Primary 14T15; Secondary 05B35, 05C30, 52B40}

\makeatother

\providecommand{\corollaryname}{Corollary}
\providecommand{\definitionname}{Definition}
\providecommand{\examplename}{Example}
\providecommand{\lemmaname}{Lemma}
\providecommand{\notationname}{Notation}
\providecommand{\propositionname}{Proposition}
\providecommand{\remarkname}{Remark}
\providecommand{\theoremname}{Theorem}

\begin{document}
\title[Biconvex Polytopes and Tropical Linear Spaces]{Biconvex Polytopes and Tropical Linear Spaces}
\author{Jaeho Shin}
\dedicatory{Dedicated to Bernd Sturmfels on the occasion of his 60th birthday}
\address{Department of Mathematical Sciences, Seoul National University, Gwanak-ro
1, Gwanak-gu, Seoul 08826, South Korea}
\email{j.shin@snu.ac.kr}
\begin{abstract}
A biconvex polytope is a classical and tropical convex hull of finitely
many points. Given a biconvex polytope, for each vertex of it we construct
a \emph{directed bigraph} and a gammoid so that the collection of
the base polytopes of those gammoids is a matroid subdivision of the
hypersimplex, thereby proving a biconvex polytope arises as a cell
of a tropical linear space. Our construction provides manually feasible
guidelines for subdividing the hypersimplex into base polytopes, without
resorting to computers. We work out the rank-$4$ case as a demonstration.
We also show there is an injection from the vertices of any $\left(k-1\right)$-dimensional
biconvex polytope into the degree-$(k-1)$ monomials in $k$ indeterminates.
\end{abstract}

\maketitle
\tableofcontents{}

\section*{Introduction}

Tropical geometry is geometry over exponents of algebraic expressions.
It is naturally equipped with a pair of ``logarithmized'' addition
and multiplication which is either $\left(\min,+\right)$ or $\left(\max,+\right)$.
Whichever to choose is a matter of preference, but we need to fix
one and our playground will be \emph{min-plus algebra} with $\left(\min,+\right)$.
In precise terms ``logarithmizing'' is \emph{tropicalizing}. Tropicalized
notions are delicate and often do not conform to our classical sense.
Tropical convexity and tropical linearity are two of such. We deepen
our understanding of them by investigating a classical and tropical
convex hull of finitely many points, which we call a \emph{biconvex
polytope}.\footnote{This is also called a \emph{polytrope}, but the ``r'' in it is apt
to be blown past and cause unnecessary confusion. We call it a \emph{biconvex
polytope} because not only is it clear, but the name says it all.}

\smallskip{}
We assume the reader is familiar with matroid theory. Starting from
scratch, we show the \emph{biconvexity} is a well-defined notion first.
Then we conduct face analysis of biconvex polytopes. We define a new
graph-theoretic notion of \emph{directed bigraph} as a directed graph
with a specific ordered bipartite structure. So, two same directed
graphs with different bipartite structures are distinguished. With
this we construct three correspondences as follows:
\begin{itemize}
\item a correspondence from the vertices of any biconvex polytope to directed
bigraphs, 
\item a correspondence from directed bigraphs to gammoids, and hence
\item a correspondence from the vertices of any biconvex polytope to gammoids.
\end{itemize}
In the second correspondence, one should a priori choose a pair of
a ground set and a partition. To obtain connected matroids as the
outcome it suffices to take a little care when choosing such a pair,
see Lemma \ref{lem:connected}. Therefore we may assume the outcome
is a collection of connected matroids.\smallskip{}

For a fixed biconvex polytope, we define a map from the edges of the
biconvex polytope to the subsets of the ground set, which we call
the \emph{combinatorial log map} by which the face structure of the
base polytope of each gammoid is completely described. We prove the
collection of the base polytopes of those gammoids is a matroid subdivision
of the hypersimplex. We will then immediately see that any biconvex
polytope arises as a cell of a tropical linear space.\smallskip{}

This theory is not cohomology-based which leads to its construction
being concise and elegant. Nonetheless, one can read off and translate
back and forth cohomology-related information. For instance, we show
there is an injection from the vertices of any $\left(k-1\right)$-dimensional
biconvex polytope into the degree-$\left(k-1\right)$ monomials in
$k$ indeterminates. When the biconvex polytope has the maximum number
of vertices which is equal to the number of those monomials, the injection
becomes a bijection.\smallskip{}

Our construction provides manually feasible guidelines for subdividing
the hypersimplex into base polytopes. As a demonstration we work out
the rank-$4$ case.\smallskip{}

All the computations are manually done with pen and paper by using
our theory, without resorting to computers.

\subsection*{Terminological note}

For a finite set $S$, we denote by $\mathbb{R}^{S}$ the Cartesian
product of $\left|S\right|$ copies of $\mathbb{R}$ that are labelled
by the elements of $S$.\smallskip{}

The rank-$k$ uniform matroid on $S$ is denoted by $U_{S}^{k}$.
The base polytope of $U_{S}^{k}$ is denoted by $\Delta_{S}^{k}$
which is called a \emph{hypersimplex}. For $S=\left[n\right]:=\left\{ 1,\dots,n\right\} $,
we write $U_{n}^{k}$ for $U_{\left[n\right]}^{k}$, and $\Delta_{n}^{k}$
for $\Delta_{\left[n\right]}^{k}$. We will often write $\Delta$
without a superscript nor a subscript unless confusion could arise.\smallskip{}

For all $i\in S$ we understand $x_{i}$ as coordinate functions of
$\mathbb{R}^{S}$ or indeterminates for the coordinates. For a vector
$\mathbf{v}\in\mathbb{R}^{S}$ and $i\in S$ we denote by $x_{i}\left(\mathbf{v}\right)$
the $i$-th coordinate of $\mathbf{v}$. For a nonempty subset $A$
of $S$, we denote ${\textstyle x\left(A\right)=\sum_{i\in A}x_{i}}$.\smallskip{}

Let $Q$ be a polyhedron in $\mathbb{R}^{S}$ and $\mathcal{Q}$ a
set of describing equations and inequalities of it. We frequently
write $\mathcal{Q}$ for $Q$ (even though there can be different
such sets). So, $\left\{ x\left(S\right)=k\right\} $ may denote the
polyhedron in an ambient space determined by the equation $x\left(S\right)=k$.
For instance, $\Delta_{S}^{k}=\left\{ x\left(S\right)=k\right\} =\left[0,1\right]^{S}\cap\left\{ x\left(S\right)=k\right\} $.\smallskip{}

We deal with directed graphs in Section \ref{sec:Graphical Model}
for which we use the terms ``nodes'' and ``arrows'' instead of
``vertices'' and ``edges'' to avoid any possible confusion because
we use the latter terms for biconvex polytopes. But, arrows without
directions can still be referred to as edges.\medskip{}

We use \textbf{boldface} to define terms and \emph{italics} for emphasis.

\subsection*{Acknowledgements}

The author learned of the possible relationship between tropical linear
spaces and biconvex polytopes from Bernd Sturmfels in 2014 and started
this research at his suggestion. Special thanks to Thomas Zaslavsky
for invaluable conversations and comments. He pointed out that those
matroids constructed in Section \ref{sec:Gammoid} are gammoids. The
author is grateful to June Huh and Günter Ziegler for their interest
and some advice. He would also like to acknowledge email exchanges
with Michael Joswig, Benjamin Schröter, and David Speyer.

This research, at the final stage, was partially supported by the
fund of National Research Foundation of Korea (\#2019R1A2C3010487),
and he thanks JongHae Keum for the support.

\section{\label{sec:Face Analysis}Face Analysis of Biconvex Polytopes}

We need to fix notation first. For $a,b\in\mathbb{R}$, the \textbf{tropical
sum} of $a$ and $b$ is $a\boxplus b=\min\left\{ a,b\right\} $\footnote{We use $\boxplus$ to denote tropical sum as reserving $\oplus$ for
direct sum.} and the \textbf{tropical product} of $a$ and $b$ is $a\odot b=a+b$.
Then, $\left(\mathbb{R}\cup\left\{ \infty\right\} ,\boxplus,\odot\right)$
is a semiring which is called \textbf{min-plus algebra}. By replacing
minimum with maximum, we obtain another semiring $\left(\mathbb{R}\cup\left\{ -\infty\right\} ,\max,+\right)$
which is called \textbf{max-plus algebra}. These two algebras are
isomorphic.

\subsection{Biconvex polytopes}

Let $k$ be a positive integer and $\mathbbm1$ denote the all-one
vector $\left(1,\dots,1\right)\in\mathbb{R}^{k}$. For $a\in\mathbb{R}$
and $\mathbf{x}\in\mathbb{R}^{k}$, the \textbf{tropical scalar multiplication}
 $a\odot\mathbf{x}$ is 
\[
a\odot\mathbf{x}=a\mathbbm1+\mathbf{x}.
\]

For vectors $\mathbf{x}_{1},\dots,\mathbf{x}_{n}\in\mathbb{R}^{k}$,
their tropical sum $\mathbf{x}_{1}\boxplus\cdots\boxplus\mathbf{x}_{n}$
is entrywise defined, that is, the $j$-th entry of $\mathbf{x}_{1}\boxplus\cdots\boxplus\mathbf{x}_{n}$
is $x_{1j}\boxplus\cdots\boxplus x_{nj}$ where $x_{ij}$ is the $j$-th
entry of $\mathbf{x}_{i}$. A \textbf{tropical linear sum} or a \textbf{tropical
linear combination} of $\mathbf{x}_{1},\dots,\mathbf{x}_{n}$ is
\[
a_{1}\odot\mathbf{x}_{1}\boxplus\cdots\boxplus a_{n}\odot\mathbf{x}_{n}
\]
 for some $a_{1},\dots,a_{n}\in\mathbb{R}$.\smallskip{}

A subset of $\mathbb{R}^{k}$ is called \textbf{tropically convex}
if it is closed under the operation of tropical linear sum. The \textbf{tropical
convex hull} of a subset $V\subset\mathbb{R}^{k}$ is the smallest
tropically convex subset that contains $V$, denoted by $\mathrm{tconv}\left(V\right)$.
Here, we say $V$ \textbf{generates} $\mathrm{tconv}\left(V\right)$.\smallskip{}

A tropically convex subset in $\mathbb{R}^{k}$ is closed under tropical
scalar multiplication. Thus, tropical convex hull is well-defined
over the quotient space $\mathbb{R}^{k}/\mathbb{R}\mathbbm1$ although
an individual tropical linear sum is \emph{not}.\smallskip{}

A tropically convex subset in $\mathbb{R}^{k}$ is an \emph{unbounded}
polyhedron. But, it is \emph{bounded} in $\mathbb{R}^{k}/\mathbb{R}\mathbbm1$
if it is generated by finitely many points. This leads us to the following
definition.
\begin{defn}
A \textbf{biconvex polytope} is a convex polytope in $\mathbb{R}^{k}/\mathbb{R}\mathbbm1$
that is a tropical convex hull of finitely many points.
\end{defn}

\begin{notation}
For two points $\mathbf{a},\mathbf{b}\in\mathbb{R}^{k}$, we define:
\noindent \begin{center}
$\mathbf{a}\equiv\mathbf{b}$\quad{}if and only if\quad{}$\mathbf{a}-\mathbf{b}=t\cdot\mathbbm1$
for some real number $t$.
\par\end{center}

\noindent Then, ``$\equiv$'' is equality in $\mathbb{R}^{k}/\mathbb{R}\mathbbm1$,
that is, equality in $\mathbb{R}^{k}$ modulo $\mathbb{R}\mathbbm1$.
\end{notation}

\subsection{Maximal biconvex polytopes}

Given a biconvex polytope, by passing to the tropical projective space
of the same dimension if necessary, we may assume it is full-dimensional,
cf. \cite[Proposition 17]{DS04}.\medskip{}

Let $\mathrm{tconv}\left(\mathbf{v}_{1},\dots,\mathbf{v}_{k}\right)\subset\mathbb{R}^{k}/\mathbb{R}\mathbbm1$
be full-dimensional, then it contains a unique full-dimensional cell
$P$, and henceforth we may assume: 
\[
P=\mathrm{tconv}\left(\mathbf{v}_{1},\dots,\mathbf{v}_{k}\right).
\]
 Then, $\left\{ \mathbf{v}_{1},\dots,\mathbf{v}_{k}\right\} $ is
the unique inclusionwise minimal set of points in $\mathbb{R}^{k}/\mathbb{R}\mathbbm1$
that generates $P$ where the inclusion is set inclusion, \cite[Proposition 21]{DS04}.
Denote 
\[
\mathrm{Vert}\left(P\right)=\left\{ \text{the vertices of }P\right\} .
\]
 The cardinality of $\mathrm{Vert}\left(P\right)$ is at least $k$
and at most $\binom{2k-2}{k-1}$, \cite[Proposition 19]{DS04}, and
it makes sense to introduce the following notation 
\[
\mathrm{Vert}^{0}\left(P\right):=\mathrm{Vert}\left(P\right)-\left\{ \mathbf{v}_{1},\dots,\mathbf{v}_{k}\right\} .
\]

\begin{defn}
A \textbf{maximal} biconvex polytope is a full-dimensional one with
the maximum number of vertices.
\end{defn}

Unless otherwise stated we assume our biconvex polytope is maximal
because a biconvex polytope of lower dimension or with fewer number
of vertices is obtained as a tropical degeneration of a maximal biconvex
polytope $\mathrm{tconv}\left(\mathbf{v}_{1},\dots,\mathbf{v}_{k}\right)$
for some integer $k$ as varying the $k$ points $\mathbf{v}_{1},\dots,\mathbf{v}_{k}$.

\subsection{Min-plus hyperplanes}

The \textbf{min-plus hyperplane} $H^{\mathbf{a}}$ at $\mathbf{a}=\left(a_{1},\dots,a_{k}\right)\in\mathbb{R}^{k}$
is defined as the set of points $\left(x_{1},\dots,x_{k}\right)$
such that the minimum 
\[
a_{1}\odot x_{1}\boxplus\cdots\boxplus a_{k}\odot x_{k}
\]
 occurs at least twice, that is, the minimum equals $a_{i}\odot x_{i}$
and $a_{j}\odot x_{j}$ for some $i$ and $j$ with $i\neq j$. Then,
$H^{\mathbf{a}}$ is written as
\[
H^{\mathbf{a}}:=\bigcup_{i,j\in\left[k\right],\,i\neq j}\;\bigcap_{l\in\left[k\right]-\left\{ i,j\right\} }\left\{ a_{l}+x_{l}\ge a_{i}+x_{i}=a_{j}+x_{j}\right\} .
\]
Min-plus hyperplanes are sometimes called \textbf{tropical hyperplanes}.
We call each summand ${\displaystyle \bigcap_{l\in\left[k\right]-\left\{ i,j\right\} }\left\{ a_{l}+x_{l}\ge a_{i}+x_{i}=a_{j}+x_{j}\right\} }$
with $i\neq j$ the \textbf{\{}\textbf{\emph{i,\,j}}\textbf{\}-min-branch}
of the tropical hyperplane. The \textbf{\emph{i}}\textbf{-th} \textbf{min-sector}
by the tropical hyperplane \textbf{at} $\mathbf{a}$ is defined as
\[
E_{i}^{\mathbf{a}}:=\bigcap_{j\in\left[k\right]-\left\{ i\right\} }\left\{ a_{j}+x_{j}\ge a_{i}+x_{i}\right\} .
\]
Min-plus hyperplanes, min-branches, and min-sectors are tropically
convex, cf.~\cite[Proposition 6 and Corollary 7]{DS04}.

\subsection{Max-plus hyperplanes}

In the same manner, the \textbf{max-plus hyperplane $\bar{H}^{\mathbf{a}}$}
\textbf{at} $\mathbf{a}=\left(a_{1},\dots,a_{k}\right)$ is defined
as the set of points $\left(x_{1},\dots,x_{k}\right)$ such that the
maximum 
\[
\max\left(a_{1}+x_{1},\dots,a_{k}+x_{k}\right)
\]
 occurs at least twice, which is 
\[
\bar{H}^{\mathbf{a}}=\bigcup_{i,j\in\left[k\right],\,i\neq j}\;\bigcap_{l\in\left[k\right]-\left\{ i,j\right\} }\left\{ a_{l}+x_{l}\le a_{i}+x_{i}=a_{j}+x_{j}\right\} .
\]
 We call each summand ${\displaystyle \bigcap_{l\in\left[k\right]-\left\{ i,j\right\} }\left\{ a_{l}+x_{l}\le a_{i}+x_{i}=a_{j}+x_{j}\right\} }$
with $i\neq j$ the \textbf{\{}\textbf{\emph{i,\,j}}\textbf{\}-max-branch}
of the max-plus hyperplane. The \textbf{\emph{i}}\textbf{-th} \textbf{max-sector}
by the max-plus hyperplane \textbf{at} $\mathbf{a}$ is defined as
\[
\bar{E}_{i}^{\mathbf{a}}:=\bigcap_{j\in\left[k\right]-\left\{ i\right\} }\left\{ a_{j}+x_{j}\le a_{i}+x_{i}\right\} .
\]
 Max-sectors and min-sectors are related as follows: for $i\in\left[k\right]$
and $\mathbf{a}\in\mathbb{R}^{k}$
\[
\bar{E}_{i}^{\mathbf{a}}=-E_{i}^{\left(-\mathbf{a}\right)}.
\]
Max-plus hyperplanes are closed under tropical scalar multiplication
and so are max-branches and max-sectors.

\subsection{Tropical objects over the quotient space $\mathbb{R}^{k}/\mathbb{R}\mathbbm1$.}

All of min-plus and max-plus hyperplanes, branches, and sectors are
well-defined over $\mathbb{R}^{k}/\mathbb{R}\mathbbm1$. Any branch
has codimension $1$, and the boundary of a sector is a union of branches.
Thus we will often call the $\left\{ i,j\right\} $-branch of a sector
the \textbf{\{}\textbf{\emph{i,\,j}}\textbf{\}-facet} of the sector.

\subsection{Vertices of biconvex polytopes}

Let $P=\mathrm{tconv}\left(\mathbf{v}_{1},\dots,\mathbf{v}_{k}\right)\subset\mathbb{R}^{k}/\mathbb{R}\mathbbm1$
be a maximal biconvex polytope, then $\left\{ \mathbf{v}_{1},\dots,\mathbf{v}_{k}\right\} \subset\mathrm{Vert}\left(P\right)$
generates $P$. At each vertex $\mathbf{v}_{i}$ of $P$ there is
a unique max-sector that contains $P$, say $\bar{E}_{\pi\left(i\right)}^{\mathbf{v}_{i}}$
for some permutation $\pi:\left[k\right]\rightarrow\left[k\right]$.
From now on, by rearranging indices if necessary, we may assume $\pi=\mathrm{id}$
and write 
\[
\bar{E}_{i}:=\bar{E}_{i}^{\mathbf{v}_{i}}.
\]
 Then:
\[
P=\bigcap_{i\in\left[k\right]}\bar{E}_{i}.
\]

\begin{defn}
For any vertex $\mathbf{w}\in\mathrm{Vert}\left(P\right)$, its \textbf{index
set} $I^{\mathbf{w}}$ is defined as 
\[
I^{\mathbf{w}}:=\left\{ i\in\left[k\right]:\mathbf{w}\in\left(\text{a facet of }\bar{E}_{i}\right)\right\} \neq\emptyset.
\]
 For each index $i\in I^{\mathbf{w}}$, the \textbf{exponent set}
$C_{i}^{\mathbf{w}}$ and the \textbf{span set} $D_{i}^{\mathbf{w}}$
are defined as 
\begin{align*}
C_{i}^{\mathbf{w}} & :=\left\{ l\in\left[k\right]-\left\{ i\right\} :\mathbf{w}\in\left(\text{the }\left\{ i,l\right\} \text{-facet of }\bar{E}_{i}\right)\right\} ,\\
D_{i}^{\mathbf{w}} & :=\left[k\right]-\left\{ i\right\} -C_{i}^{\mathbf{w}}.
\end{align*}
 Then, $i\notin C_{i}^{\mathbf{w}}$ and $i\notin D_{i}^{\mathbf{w}}$.
In particular, for $\mathbf{w}\in\left\{ \mathbf{v}_{1},\dots,\mathbf{v}_{k}\right\} $,
say for $\mathbf{w}=\mathbf{v}_{j}$:
\begin{align*}
I^{\mathbf{v}_{j}} & =\left\{ j\right\} ,\\
C_{j}^{\mathbf{v}_{j}} & =\left[k\right]-\left\{ j\right\} ,\\
D_{j}^{\mathbf{v}_{j}} & =\emptyset.
\end{align*}
 We will often write $I$, $C_{i}$ and $D_{i}$ without the superscript
$^{\mathbf{w}}$ for simplicity unless confusion could arise.
\end{defn}

Let $\mathbf{e}_{1},\dots,\mathbf{e}_{k}$ be the standard basis vectors
of $\mathbb{R}^{k}$. For two $i,j\in I$, there are nonnegative real
numbers $a_{l}$ for $l\in D_{i}$ and $b_{m}$ for $m\in D_{j}$
with 
\[
\mathbf{w}\equiv\mathbf{v}_{i}+\sum_{l\in D_{i}}a_{l}\left(-\mathbf{e}_{l}\right)\equiv\mathbf{v}_{j}+\sum_{m\in D_{j}}b_{m}\left(-\mathbf{e}_{m}\right).
\]
Note that we employ two affine coordinate systems here, one for each
of $\mathbf{v}_{i}$ and $\mathbf{v}_{j}$. Because the maximum of
coordinates of $\mathbf{w}-\mathbf{v}_{i}$ is its $i$-th coordinate,
$i\in D_{j}$ with $b_{i}\neq0$, and similarly $j\in D_{i}$ with
$a_{j}\neq0$. Thus, $I\subseteq D_{i}\cup\left\{ i\right\} $ for
any $i\in I$, and 
\begin{equation}
I\cap C_{i}=\emptyset.\label{eq:disjointness}
\end{equation}
 The vertex $\mathbf{w}$ is an intersection of $k-1$ branches, each
of which contains one and only one vertex from $\left\{ \mathbf{v}_{i}:i\in I\right\} $
by maximality of $P$. More specifically:
\begin{equation}
\mathbf{w}=\bigcap_{i\in I}\bigcap_{\;l\in C_{i}}\left(\text{the }\left\{ i,l\right\} \text{-facet of }\bar{E}_{i}\right).\label{eq:intersection}
\end{equation}
 The operand ${\displaystyle \bigcap_{l\in C_{i}}\left(\text{the }\left\{ i,l\right\} \text{-facet of }\bar{E}_{i}\right)}$
for $i\in I$ has codimension $\left|C_{i}\right|$ and 
\begin{equation}
\sum_{i\in I}\left|C_{i}\right|=k-1.\label{eq:codim}
\end{equation}
 A computation shows: 
\[
\bigcup_{i\in I}C_{i}=\left[k\right]-I\quad\text{and}\quad\bigcap_{i\in I}D_{i}=\emptyset.
\]
 If $\mathbf{w}\in\mathrm{Vert}^{0}\left(P\right)$, then $\left|I^{\mathbf{w}}\right|\ge2$
and vice versa, and we have 
\begin{equation}
0\le\left|\bigcap_{i\in I}C_{i}\right|\le1\quad\text{and}\quad k-1\le\left|\bigcup_{i\in I}D_{i}\right|\le k.\label{eq:type}
\end{equation}
 Now, for any vertex $\mathbf{w}\in\mathrm{Vert}\left(P\right)$ and
for all $i\in\left[k\right]-I$ define $C_{i}:=\emptyset$ and $D_{i}:=\emptyset.$
Then, $C_{i}$ and $D_{i}$ are  defined for all $i\in\left[k\right]$.
Note that $C_{i}\neq\emptyset$ if and only if $i$ is an index, that
is, $i\in I$.
\begin{notation}
\label{nota:vertices}For a vertex $\mathbf{w}\in\mathrm{Vert}\left(P\right)$,
we introduce the following notation 
\[
\mathbf{w}=\mathbf{v}_{1}^{C_{1}}\cdots\mathbf{v}_{k}^{C_{k}}=\prod_{i\in\left[k\right]}\mathbf{v}_{i}^{C_{i}}.
\]
 Practically, we remove all $\mathbf{v}_{i}^{C_{i}}$ with $C_{i}=\emptyset$
from the notation and write 
\[
\mathbf{w}=\prod_{i\in I}\mathbf{v}_{i}^{C_{i}}.
\]
 In particular, $\mathbf{v}_{i}=\mathbf{v}_{i}^{\left[k\right]-\left\{ i\right\} }$
for a vertex $\mathbf{v}_{i}$.
\end{notation}

\begin{defn}
Along the formula (\ref{eq:type}) we define the \textbf{type} of
a vertex $\mathbf{w}\in\mathrm{Vert}^{0}\left(P\right)$ as follows:
\[
\mathrm{type}\left(\mathbf{w}\right)=\begin{cases}
0 & \text{if }\bigcap_{i\in I^{\mathbf{w}}}C_{i}^{\mathbf{w}}=\emptyset,\\
1 & \text{otherwise.}
\end{cases}
\]
\end{defn}

\begin{example}
If $k\le4$, every vertex $\mathbf{w}\in\mathrm{Vert}^{0}\left(P\right)$
has type $1$.
\end{example}

\subsection{Faces of biconvex polytopes}

Let $Q$ be a face of $P$. For $i\in\left[k\right]$, denote 
\[
\mathrm{Vert}_{i}\left(Q\right):=\left\{ \mathbf{w}\in\mathrm{Vert}\left(Q\right):C_{i}^{\mathbf{w}}\neq\emptyset\right\} .
\]

\begin{defn}
The \textbf{index set} of $Q$ is defined as 
\[
I^{Q}:=\left\{ i\in\left[k\right]:\mathrm{Vert}_{i}\left(Q\right)\neq\emptyset\right\} .
\]
 For each index $i\in I^{Q}$, the \textbf{exponent set} $C_{i}^{Q}$
and the \textbf{span set} $D_{i}^{Q}$ are defined as 
\begin{align*}
C_{i}^{Q} & :={\displaystyle \bigcap_{\mathbf{w}\in\mathrm{Vert}_{i}(Q)}C_{i}^{\mathbf{w}}},\\
D_{i}^{Q} & :=\left[k\right]-\left\{ i\right\} -C_{i}^{Q}.
\end{align*}
 Using the above definitions, $Q$ is written as the following: 
\[
Q=\bigcap_{i\in I^{Q}}\bigcap_{l\in C_{i}^{Q}}\left(\text{the }\left\{ i,l\right\} \text{-facet of }\bar{E}_{i}\right)\cap P.
\]
 Moreover, its codimension is 
\[
\mathrm{codim}(Q)=\sum_{i\in I^{Q}}\left|C_{i}^{Q}\right|.
\]
 Observe that these are consistent with (\ref{eq:intersection}) and
(\ref{eq:codim}). Note that $I^{Q}\neq\emptyset$ and $C_{i}^{Q}\neq\emptyset$
for all $i\in I^{Q}$.\medskip{}
\end{defn}

Similarly as in the vertex case, for all $i\in\left[k\right]-I^{Q}$
define $C_{i}^{Q}:=\emptyset$ and $D_{i}^{Q}:=\emptyset.$ Then,
$C_{i}^{Q}$ and $D_{i}^{Q}$ are defined for all $i\in\left[k\right]$.
Note that $C_{i}^{Q}\neq\emptyset$ if and only if $i$ is an index,
that is, $i\in I^{Q}$.\medskip{}

Thus, we generalized the notions for vertices. We also generalize
Notation~\ref{nota:vertices}.
\begin{notation}
\label{nota:Faces}For a face $Q$ of $P$, we denote 
\[
Q=\mathbf{v}_{1}^{C_{1}^{Q}}\cdots\mathbf{v}_{k}^{C_{k}^{Q}}=\prod_{i\in\left[k\right]}\mathbf{v}_{i}^{C_{i}^{Q}}.
\]
\end{notation}

\begin{rem}
\label{rem:Expression}Notation~\ref{nota:Faces} describes how to
obtain a face $Q$ of $P$ from the unique generating set of vertices
$\left\{ \mathbf{v}_{1},\dots,\mathbf{v}_{k}\right\} $. It further
generalizes to a biconvex polytope that is \emph{not} maximal, but
then the uniqueness of expression fails.
\end{rem}

\section{\label{sec:Graphical Model}Graphical Model and Monomial Map}

\subsection{Directed bigraphs and faces of biconvex polytopes}

A \emph{bipartite graph} or a \emph{bigraph} for short is a graph
that does not contain any odd cycle.
\begin{defn}
A \textbf{directed bigraph} $G$ is a bigraph with \emph{ordered}
parts $\left(I,I^{c}\right)$ for a subset $I\subseteq V_{G}$ such
that if $\left(i,c\right)$ is an arrow, then $i\in I$ and $c\in I^{c}$.\footnote{A directed bigraph is a directed graph. We do not allow multiple arrows.}
\end{defn}

By definition, a directed bigraph is a directed graph with its own
ordered bipartite structure. So, two same directed graphs with different
ordered bipartite structures are distinguished. Note that one of $I$
and $I^{c}$ can be empty (in this case the graph is a set of isolated
nodes), but not both of them can.\medskip{}

Let $P=\mathrm{tconv}\left(\mathbf{v}_{1},\dots,\mathbf{v}_{k}\right)\subset\mathbb{R}^{k}/\mathbb{R}\mathbbm1$
be a maximal biconvex polytope. To each vertex $\mathbf{w}$ of $P$,
assign the directed graph $G_{\mathbf{w}}$ with node set $\left[k\right]$
satisfying that
\[
\left(i,c\right)\text{ is an arrow of }G_{\mathbf{w}}\quad\text{if and only if}\quad i\in I^{\mathbf{w}}\text{ and }c\in C_{i}^{\mathbf{w}}.
\]
 Then, $G_{\mathbf{w}}$ is a directed bigraph with ordered parts
$\left(I^{\mathbf{w}},\left[k\right]-I^{\mathbf{w}}\right)$ by (\ref{eq:disjointness}).
Moreover, $G_{\mathbf{w}}$ is a \emph{tree} by (\ref{eq:codim}).\medskip{}

Let $G$ be a directed bigraph with $V_{G}=\left[k\right]$ that is
a forest. If $i\in I$ and $c\in I^{c}$ are adjacent, then denote
by $\left(i,c\right)$ the \emph{unique} arrow that connects $i$
to $c$, and denote by $G(i,c)$ the graph obtained from $G$ by removing
$\left(i,c\right)$ from it. Then, $G(i,c)$ is a directed bigraph
with induced bipartite structure.

\medskip{}

Let $G^{+}(i,c)$ and $G^{-}(i,c)$ be the connected components of
$G(i,c)$ containing $i$ and $c$, respectively, both of which are
nonempty and have induced bipartite structure. Again, $G^{+}(i,c)$
and $G^{-}(i,c)$ are connected by definition whether or not $G$
is connected, and two node sets $V_{G^{+}(i,c)}$ and $V_{G^{-}(i,c)}$
are disjoint, which partition the node set of the connected component
of $G$ that contains $\left(i,c\right)$.\medskip{}

Let $Q$ be a nonempty proper face of $P$. Define a directed graph
$G_{Q}$ with node set $\left[k\right]$  such that
\[
\left(i,c\right)\text{ is an arrow of }G_{Q}\quad\text{if and only if}\quad i\in I^{Q}\text{ and }c\in C_{i}^{Q}.
\]
Given a vertex $\mathbf{w}$ of $Q$, we have 
\[
Q=\prod_{i\in\left[k\right]}\mathbf{v}_{i}^{C_{i}^{\mathbf{w}}-\left(C_{i}^{\mathbf{w}}-C_{i}^{Q}\right)}.
\]
I.e. $G_{Q}$ is obtained from $G_{\mathbf{w}}$ by deleting arrows
$\left(i,c\right)$ with $i\in I^{\mathbf{w}}$ and $c\in C_{i}^{\mathbf{w}}-C_{i}^{Q}$.
Therefore $G_{Q}$ is a directed bigraph that is a forest. Its bipartite
structure is induced from that of $G_{\mathbf{w}}$, but does not
depend on the choice of the vertex. The number of its connected components
minus the number of its isolated nodes is 
\[
\dim(Q)=\sum_{i\in I^{\mathbf{w}}}\left|C_{i}^{\mathbf{w}}-C_{i}^{Q}\right|.
\]

\subsection{\label{subsec:Edges}Edges of biconvex polytopes}

Let $Q$ be an edge of $P$ with vertices $\left\{ \mathbf{w}_{1},\mathbf{w}_{2}\right\} $.
There are arrows $\left(i_{1},c_{1}\right)$ and $\left(i_{2},c_{2}\right)$
of $G_{\mathbf{w}_{1}}$ and $G_{\mathbf{w}_{2}}$, respectively,
such that 
\[
Q=\left(\prod_{i\in I^{\mathbf{w}_{1}}-\left\{ i_{1}\right\} }\mathbf{v}_{i}^{C_{i}^{\mathbf{w}_{1}}}\right)\mathbf{v}_{i_{1}}^{C_{i_{1}}^{\mathbf{w}_{1}}-\left\{ c_{1}\right\} }=\left(\prod_{i\in I^{\mathbf{w}_{2}}-\left\{ i_{2}\right\} }\mathbf{v}_{i}^{C_{i}^{\mathbf{w}_{2}}}\right)\mathbf{v}_{i_{2}}^{C_{i_{2}}^{\mathbf{w}_{2}}-\left\{ c_{2}\right\} }
\]
 where $i_{l}\in I^{\mathbf{w}_{l}}$ and $c_{l}\in C_{i_{l}}^{\mathbf{w}_{l}}$
for $l=1,2$. In other words:
\[
G_{Q}=G_{\mathbf{w}_{1}}(i_{1},c_{1})=G_{\mathbf{w}_{2}}(i_{2},c_{2}).
\]
Let $G_{1}'=G_{\mathbf{w}_{1}}^{+}(i_{1},c_{1})$ and $G_{2}'=G_{\mathbf{w}_{1}}^{-}(i_{1},c_{1})$,
then $G_{1}'$ and $G_{2}'$ are the two connected components of $G_{Q}$
with 
\[
G_{Q}=G_{1}'\cup G_{2}'.
\]

We show  $i_{2}\in V_{G_{2}'}$. For $l=1,2$, let $I_{l}=I^{\mathbf{w}_{1}}\cap V_{G_{l}'}$
and $C^{l}=V_{G_{l}'}-I_{l}$, then $G_{1}'$ and $G_{2}'$ are directed
bigraphs with parts $\left(I_{1},C^{1}\right)$ and $\left(I_{2},C^{2}\right)$,
respectively, and $G_{Q}$ is a directed bigraph with parts $\left(I_{1}\cup I_{2},C^{1}\cup C^{2}\right)$.

\smallskip{}

Let $Q_{1}$ and $Q_{2}$ be the faces of $P$ corresponding to two
directed bigraphs $G_{1}'\cup V_{G_{2}'}$ and $G_{2}'\cup V_{G_{1}'}$,
respectively, then 
\[
Q=Q_{1}\cap Q_{2}.
\]

Every point of $Q_{1}$ is contained in the $\left\{ l,c\right\} $-max-branch
of the max-plus hyperplane at $\mathbf{v}_{l}$ for all $l\in I_{1}$
and $c\in C^{1}\cap C_{l}^{Q}$. So, for each $i\in V_{G_{1}'}$ the
$i$-th coordinate of the point is a fixed real number, say $a_{i}$,
and so is the $j$-th coordinate of any point of $Q_{2}$ for each
$j\in V_{G_{2}'}$, say $b_{j}$. Note that we are employing two affine
coordinate systems. Then, given a point $\mathbf{u}$ of $Q$ there
are real numbers $x_{i}$ and $y_{j}$ for $i\in V_{G_{1}'}$ and
$j\in V_{G_{2}'}$ with:
\[
\mathbf{u}\equiv\sum_{i\in V_{G_{1}'}}a_{i}\mathbf{e}_{i}+\sum_{j\in V_{G_{2}'}}y_{j}\mathbf{e}_{j}\equiv\sum_{i\in V_{G_{1}'}}x_{i}\mathbf{e}_{i}+\sum_{j\in V_{G_{2}'}}b_{j}\mathbf{e}_{j}.
\]
Thus, there is a fixed real number $t$ with $x_{i}=a_{i}+t$ and
$y_{j}=b_{j}-t$ for all $i\in V_{G_{1}'}$ and $j\in V_{G_{2}'}$.
Let $\mathbf{s}=\sum_{i\in V_{G_{1}'}}a_{i}\mathbf{e}_{i}+\sum_{j\in V_{G_{2}'}}b_{j}\mathbf{e}_{j}$,
then 
\[
\mathbf{u}\equiv\mathbf{s}-t\cdot1^{V_{G_{2}'}}\equiv\mathbf{s}+t\cdot1^{V_{G_{1}'}}.
\]
Now, since $i_{1}\in V_{G_{1}'}$, we have $i_{2}\in V_{G_{2}'}$.
Moreover, the vector $\overrightarrow{\mathbf{w}_{1}\mathbf{w}_{2}}$
is a positive multiple of $-1^{V_{G_{2}'}}$ in $\mathbb{R}^{k}/\mathbb{R}\mathbbm1$
and $\overrightarrow{\mathbf{w}_{2}\mathbf{w}_{1}}$ is that of $-1^{V_{G_{1}'}}.$
In particular: 
\[
G_{\mathbf{w}_{1}}^{+}(i_{1},c_{1})=G_{\mathbf{w}_{2}}^{-}(i_{2},c_{2})\quad\text{and}\quad G_{\mathbf{w}_{1}}^{-}(i_{1},c_{1})=G_{\mathbf{w}_{2}}^{+}(i_{2},c_{2}).
\]

\subsection{Combinatorial log map}
\begin{defn}
\label{def:comb-log} Let $P=\mathrm{tconv}\left(\mathbf{v}_{1},\dots,\mathbf{v}_{k}\right)\subset\mathbb{R}^{k}/\mathbb{R}\mathbbm1$
be a maximal biconvex polytope. For each vertex $\mathbf{w}$ of $P$,
there are exactly $k-1$ edges $Q$ of $P$. For each edge $Q$ with
$\mathrm{Vert}\left(Q\right)=\left\{ \mathbf{w},\mathbf{v}\right\} $,
there is the unique subset $V$ of $\left[k\right]$ with: 
\[
\mathbf{v}\equiv\mathbf{w}-t\cdot1^{V}
\]
 for a \emph{positive} number $t$ which also uniquely exists. We
define $\mathrm{L}^{\mathbf{w}}$ such that 
\[
\begin{array}{ccc}
\mathrm{L}^{\mathbf{w}}:\left\{ k-1\text{ edges }Q\text{ of }P\text{ containing }\mathbf{w}\right\}  & \rightarrow & 2^{\left[k\right]}\\
\qquad Q & \mapsto & V
\end{array}.
\]
We call $\mathrm{L}$ the \textbf{combinatorial log map}\footnote{This is named after the logarithmic map producing amoebas, cf.~\cite[Chapter 6.1.B]{GKZ94}.}
and $\mathrm{L}^{\mathbf{w}}$ the \textbf{combinatorial log map}
\textbf{at $\mathbf{w}$} for $P$.
\end{defn}

Note that $\emptyset\neq\mathrm{L}^{\mathbf{w}}(Q)\neq\left[k\right]$
and $\mathrm{L}^{\mathbf{v}}(Q)=\left[k\right]-\mathrm{L}^{\mathbf{w}}(Q)$.\medskip{}

Let $P'$ be a tropical degeneration of $P$, which is also biconvex,
but not necessarily maximal. Because the direction vectors of edges
of $P'$ are direction vectors of edges of $P$, the combinatorial
log map is defined for $P'$. Thus, the combinatorial log map is defined
for \emph{any} biconvex polytope.
\begin{example}
The edge structure of $P$ at a type-$1$ vertex $\mathbf{w}\in\mathrm{Vert}^{0}(P)$
is particularly nice because if $\left|\bigcap_{i\in I^{\mathbf{w}}}C_{i}^{\mathbf{w}}\right|=1$,
then all subsets $C_{j}^{\mathbf{w}}-\bigcap_{i\in I^{\mathbf{w}}}C_{i}^{\mathbf{w}}$
of $\left[k\right]$ with $j\in I^{\mathbf{w}}$ are mutually disjoint.
The directed bigraph $G_{\mathbf{w}}$ also has a nice structure,
see Figure \ref{fig:Edges-type1}.\footnote{We draw a directed bigraph such that its parts are $\left(I=\left\{ \text{lower nodes}\right\} ,I^{c}=\left\{ \text{upper nodes}\right\} \right)$.}
It is easy to find all $k-1$ edges $Q$ connected to $\mathbf{w}$,
whose images under $\mathrm{L}^{\mathbf{w}}$ are 
\[
\mathrm{L}^{\mathbf{w}}(Q)=\begin{cases}
C_{j}^{\mathbf{w}}-\bigcap_{i\in I^{\mathbf{w}}}C_{i}^{\mathbf{w}} & \text{for }j\in I^{\mathbf{w}},\\
\left[k\right]-\left\{ j\right\}  & \text{for }j\in\left[k\right]-\bigcap_{i\in I^{\mathbf{w}}}C_{i}^{\mathbf{w}}-I^{\mathbf{w}}.
\end{cases}
\]
\begin{figure}[th]
\noindent \centering{}\noindent \begin{center}
\begin{tikzpicture}[blowup line/.style={line width=2pt,grey},font=\scriptsize]

\begin{scope}[line cap=round,scale=1.2,decoration={markings,mark=at position 0.65 with {\arrow{Stealth}}}]
 \fill [black](0,1) circle(1.83pt) coordinate(a0)node[above=0.7em]{$\bigcap_{i\in I^{\mathbf{w}}}C_{i}^{\mathbf{w}}$};
 \draw [densely dashed,ultra thin] (0,1) ellipse (0.2 and 0.2);

\foreach \x in {1,2}{
   \path (2.5*\x,0)coordinate(a\x)node[right]{$i_{\x}$}
  ++(0,1)node[above=0.7em]{$C_{i_{\x}}^{\mathbf{w}}-\bigcap_{i\in I^{\mathbf{w}}}C_{i}^{\mathbf{w}}$};
 \draw [densely dashed,ultra thin] (2.5*\x,1) ellipse (1 and 0.2);}

 \path (8.5,0) coordinate(a4)node[right]{$i_m$}
   ++(0,1)node[above=0.7em]{$C_{i_{m}}^{\mathbf{w}}-\bigcap_{i\in I^{\mathbf{w}}}C_{i}^{\mathbf{w}}$};
 \draw [densely dashed,ultra thin] (8.5,1) ellipse (1 and 0.2);

 \foreach \x in {-1,0,1}{
   \fill (6.75+.3*\x,.5) circle(.58pt);}

 \foreach \x in {1,2,4}{
   \fill [black](a\x) circle (1.83pt);
   \foreach \y in {-1,0,1}{
     \draw [thick,postaction={decorate}](a\x)--++(.7*\y,1);
     \fill [black](a\x)--++(.7*\y,1) circle (1.83pt);}}
 \draw [thick,postaction={decorate}] (a1)..controls(1.87,0.06)..(a0);
 \draw [thick,postaction={decorate}] (a2)..controls(2.4,-.3)..(a0);
  \draw [thick,postaction={decorate}] (a4)..controls(2.57,-.5)..(a0);

 \draw [densely dashed,ultra thin] (5.5,-.08) ellipse (4 and 0.32) node[below=1.1em]{$I^{\mathbf{w}}=\left\{i_{1},\dots,i_{m}\right\}$};

\end{scope}

\end{tikzpicture}
\par\end{center}\caption{\label{fig:Edges-type1}The directed bigraph $G_{\mathbf{w}}$ of
a type-$1$ vertex $\mathbf{w}$.}
\end{figure}
\end{example}

\subsection{Monomial map}
\begin{defn}
\label{def:Monomials}Let $k\ge2$ be an integer. Let $P=\mathrm{tconv}\left(\mathbf{v}_{1},\dots,\mathbf{v}_{k}\right)\subset\mathbb{R}^{k}/\mathbb{R}\mathbbm1$
be any full-dimensional biconvex polytope whether maximal or not.
We define a map $\mu$ on the collection of vertices of $P$ such
that 
\[
\begin{array}{l}
\mu:\left\{ \text{the vertices of }P\right\} \rightarrow\left\{ \text{the degree-}(k-1)\text{ monomials in }x_{1},\dots,x_{k}\right\} \\
\hspace{0.7cm}\mathbf{w}=\mathbf{v}_{1}^{C_{1}}\cdots\mathbf{v}_{k}^{C_{k}}\hspace{0.2cm}\mapsto\hspace{2.5cm}x_{1}^{\left|C_{1}\right|}\cdots x_{k}^{\left|C_{k}\right|}
\end{array}.
\]
 We call this map the \textbf{monomial map}.
\end{defn}

\begin{prop}
\label{prop:Monomials}The monomial map $\mu$ is injective.
\end{prop}

\begin{proof}
We prove by induction on dimension $k-1$. For all $1$-dimensional
biconvex polytopes the map $\mu$ is injective, and the base case
holds.\smallskip{}

Suppose that $\mu$ is injective for all biconvex polytopes of dimension
$\le k-1$ for some $k\ge2$. We may assume $P$ is a maximal biconvex
polytope of dimension $k$.

Let $\mathbf{w}$ and $\mathbf{v}$ be vertices of $P$ with $\mu\left(\mathbf{w}\right)=\mu\left(\mathbf{v}\right)$,
then $I^{\mathbf{w}}=I^{\mathbf{v}}=:I$.

If $\left|I\right|=1$, clearly $\mathbf{w}=\mathbf{v}$.

If $\left|I\right|\ge2$, then in a fixed affine coordinate system,
the $i$-th coordinate of $\mathbf{w}$ and the $j$-th coordinate
of $\mathbf{v}$ for all $i,j\in I$ are the same. So, $\mathbf{w}$
and $\mathbf{v}$ are contained in a \emph{proper} face of $P$. This
face is a biconvex polytope which has the unique inclusionwise minimal
generating set, cf. \cite[Proposition 21]{DS04}, and inherits its
geometry from $P$. Thus, $\mathbf{w}=\mathbf{v}$ by the induction
hypothesis.
\end{proof}
\begin{rem}
If $P$ is a \emph{maximal} biconvex polytope, the monomial map $\mu$
is a \emph{bijection} because the maximum number of vertices of $P$
equals the number of degree-$\left(k-1\right)$ monomials in $k$
indeterminates, which is $\binom{2k-2}{k-1}$.
\end{rem}

\section{\label{sec:Gammoid}Gammoids and Matroid Subdivision}

All italicized terms not defined herein shall have the same definitions
as set forth in Appendix \ref{sec:Preliminaries}.

\subsection{Gammoids of our interest}

Let $G$ be a directed graph\footnote{We do not allow multiple arrows.}
with node set $V_{G}=\left[k\right]$, and $S$ be a finite set with
a partition $S=\bigcup_{j\in\left[k\right]}S_{j}$ which we will call
an \textbf{underlying partition}. Denote by $T(j)$ the collection
of $l\in\left[k\right]$ with $\left(j,l\right)$ being an arrow of
$G$, which is possibly empty.\smallskip{}

Let $\tilde{G}$ be a directed graph with node set $V_{\tilde{G}}=\left[k\right]\sqcup S$
(the disjoint union of $\left[k\right]$ and $S$) such that $\left(j,s\right)$
is an arrow if $j\in\left[k\right]$ and $s\in S_{l}$ for some $l\in\left\{ j\right\} \cup T(j)$.\smallskip{}

We denote by $\Gamma\left[G\right]$ the gammoid obtained from $\tilde{G}$.
Let $E_{j}=\bigcup_{l\in\left\{ j\right\} \cup T(j)}S_{l}$, then
$\Gamma\left[G\right]$ is a rank-$k$ matroid on $S$ which is the
matroid union of rank-$1$ uniform matroids on $E_{j}$: 
\[
\Gamma\left[G\right]=\bigvee_{j\in\left[k\right]}U_{E_{j}}^{1}.
\]
 This is a transversal matroid. Note that $\Gamma\left[G\right]$
is defined for all directed graphs $G$.
\begin{lem}
\label{lem:connected}Let $G$ be a directed bigraph with parts $\left(I,I^{c}\right)$.
If $\left|S_{c}\right|=1$ for some $c\in I^{c}$, then $\Gamma\left[G\right]$
is disconnected. If $G$ is connected and $\left|S_{c}\right|\ge2$
for all $c\in I^{c}$, then $\Gamma\left[G\right]$ is connected.
\end{lem}

\begin{proof}
If $\left|S_{c}\right|=1$ for some $c\in I^{c}$, the singleton $S_{c}$
is a coloop of $\Gamma\left[G\right]$, and $\Gamma\left[G\right]$
is disconnected.

If $G$ is connected and $\left|S_{c}\right|\ge2$ for all $c\in I^{c}$,
we show that there is a $\left(k+1\right)$-element subset $A\subset S$
with $M|_{A}\simeq U_{k+1}^{k}$, then by Lemma \ref{lem:GMHA}(\ref{enu:inseparable})
it follows that $\Gamma\left[G\right]$ is connected. Since $G$ is
a tree, there is a partial order $<$ on its node set with the \emph{smallest}
node $i_{0}$ of degree $1$. Let $c_{0}$ be a node with $c_{0}\gtrdot i_{0}$,
i.e. $c_{0}$ covers $i_{0}$, then $c_{0}\in I^{c}$ because $G$
is a directed bigraph. Denote by $\deg_{G}j$ the degree of $j$ in
$G$.
\begin{enumerate}
\item Take an element from $S_{c}$ for each $c\in I^{c}$.
\item Take an element from $S_{i}$ for each $i\in I$ with $\deg_{G}i=1$.
\item For each $i\in I$ with $\deg_{G}i>1$, take an element from $S_{c}$
for a $c\gtrdot i$.
\item Take an element from $S_{c_{0}}$.
\end{enumerate}
Thus $k+1$ elements are taken from $S$. Let $A$ be the set of these
elements, then $\left|A\cap S_{c}\right|\le2$ for all $c\in I^{c}$
and $M|_{A}=U_{A}^{k}$. The proof is done.
\end{proof}
\begin{cor}
\label{cor:Bipartite}Let $G$ be a directed bigraph with $\left|S_{c}\right|\ge2$
for all $c\in I^{c}$. If $G$ has $m$ connected components $G_{1},\dots,G_{m}$,
then $\Gamma\left[G_{1}\right],\dots,\Gamma\left[G_{m}\right]$ are
connected and 
\[
\Gamma\left[G\right]=\Gamma\left[G_{1}\right]\oplus\cdots\oplus\Gamma\left[G_{m}\right].
\]
 In particular, the number of connected components of $G$ equals
that of $\Gamma\left[G\right]$.
\end{cor}

\begin{rem}
\label{rem:bistructure}If $G$ is an isolated node, $\Gamma\left[G\right]$
is a rank-$1$ uniform matroid. So, for a set of isolated nodes, we
often ignore its bipartite structure.
\end{rem}

\subsection{Flats of $\Gamma\left[G\right]$ for a directed bigraph $G$}

For any subgraph $G'$, denote by $S_{V_{G}'}$ or more simply by
$S_{G'}$ the set $\bigcup_{j\in V_{G'}}S_{j}$. Then, it is immediate
that $S_{G^{+}(i,c)}$ for each arrow $\left(i,c\right)$ of $G$
is a flat of $\Gamma\left[G\right]$, of rank $\left|V_{G^{+}(i,c)}\right|$.
Moreover:
\[
\Gamma\left[G^{+}(i,c)\right]=\Gamma\left[G\right]|_{S_{G^{+}(i,c)}}.
\]
Observe that any base of
\[
\Gamma\left[G(i,c)\right]=\Gamma\left[G^{+}(i,c)\right]\oplus\Gamma\left[G^{-}(i,c)\right]
\]
 is a base of $\Gamma\left[G\right]$ whose intersection with $S_{G^{+}(i,c)}$
is a base of $\Gamma\left[G^{+}(i,c)\right]$, and vice versa. Therefore
it is a base of 
\[
\Gamma\left[G\right]|_{S_{G^{+}(i,c)}}\oplus\Gamma\left[G\right]/S_{G^{+}(i,c)}
\]
 and the converse holds. Thus we obtain 
\[
\Gamma\left[G(i,c)\right]=\Gamma\left[G\right]|_{S_{G^{+}(i,c)}}\oplus\Gamma\left[G\right]/S_{G^{+}(i,c)}.
\]
 In particular:
\[
\Gamma\left[G^{-}(i,c)\right]=\Gamma\left[G\right]/S_{G^{+}(i,c)}.
\]
All this proves $S_{G^{+}(i,c)}$ is a \emph{non-degenerate} flat
of $\Gamma\left[G\right]$.\medskip{}

Since $G$ has at most $k-1$ arrows, the gammoid $\Gamma[G]$ has
at most $k-1$ \emph{minimal} non-degenerate flats of the form $S_{G^{+}(i,c)}$.\medskip{}

Then, because any intersection of flats is a flat, the sets $S_{i}$
for all $i\in I$ are flats of rank $1$. In the same way, $\bigcup_{j\in N[c]}S_{j}$
for some $c\in C$ is a flat of rank $\left|N[c]\right|$ where $N[c]$
denotes the set of $c$ and its adjacent nodes.\smallskip{}

Further, for a subset $C\subset I^{c}$ the set $\bigcup_{c\in C}\bigcup_{j\in N[c]}S_{j}$
is a flat of rank $\left|\bigcup_{c\in C}N[c]\right|$ which is connected.
This flat is non-degenerate if and only if $\bigcup_{c\in C}N[c]$
is the node set $V_{G^{+}(i_{0},c_{0})}$ of a graph $G^{+}(i_{0},c_{0})$
for some arrow $(i_{0},c_{0})$. \medskip{}

Any singleton $\left\{ s\right\} $ in $S_{c}$ for $c\in I^{c}$
is a rank-$1$ flat and of course $M|_{\left\{ s\right\} }$ is connected.
Moreover, it is a non-degenerate flat because one can choose a subset
$A\subset S$ that contains $s$ with $M|_{A}\simeq U_{k+1}^{k}$
as in Lemma \ref{lem:connected} so that $M/\left\{ s\right\} |_{A-\left\{ s\right\} }=M|_{A}/\left\{ s\right\} =U_{A-\left\{ s\right\} }^{k-1}\simeq U_{k}^{k-1}$
which implies that $M/\left\{ s\right\} $ is connected.\medskip{}

The following lemma says that any non-degenerate flat of $\Gamma\left[G\right]$
arises in a way described above.
\begin{lem}
\label{lem:nondeg-1}Let $G$ be a directed bigraph with $\left|S_{c}\right|\ge2$
for all $c\in I^{c}$. Then, a minimal non-degenerate flat of $\Gamma\left[G\right]$
is either $S_{G^{+}(i,c)}$ for an arrow $\left(i,c\right)$ of $G$
or a singleton contained in $S_{c}$ for some $c\in I^{c}$.
\end{lem}

\begin{proof}
We may assume $G$ is connected. Let $A$ be a non-degenerate flat
of $\Gamma\left[G\right]$ that is not obtained by removing an arrow
of $G$.

If $A\cap S_{c}=\emptyset$ for all $c\in I^{c}$, then $A$ is a
disjoint union of $A_{i}$ with $i\in I$ and is disconnected. So,
assume $A\cap S_{c}\neq\emptyset$ for some $c\in I^{c}$.

Then, again since $A$ is a connected flat, if $\left|A\cap S_{c}\right|>1$,
then $\bigcup_{j\in N[c]}S_{j}\subseteq A$.

If $\left|A\cap S_{c}\right|=1$ and $A\cap\bigcup_{j\in N[c]}S_{j}\neq\emptyset$,
similarly $\bigcup_{j\in N[c]}S_{j}\subseteq A$.

Therefore, $\left|A\cap S_{c}\right|=1$ and $A\cap\bigcup_{j\in N[c]}S_{j}=\emptyset$.
This implies $\left|A\right|=1$.
\end{proof}
\begin{cor}
\label{cor:nondeg-2}Let $G$ be a \emph{connected} directed bigraph
with $\left|S_{c}\right|\ge2$ for all $c\in I^{c}$. Then, there
are precisely $k-1$ facets of the \emph{base polytope} $\mathrm{BP}_{\Gamma\left[G\right]}$
of $\Gamma\left[G\right]$ that are not contained in the boundary
of the hypersimplex $\Delta=\Delta_{S}^{k}$, which are $\mathrm{BP}_{\Gamma\left[G(i,c)\right]}$
for the $k-1$ arrows $\left(i,c\right)$ of $G$.
\end{cor}

\subsection{\label{subsec:MAIN}Matroid subdivisions dual to biconvex polytopes}

Fix an integer $k\ge1$ and an underlying partition $S=\bigcup_{i\in\left[k\right]}S_{i}$.
For any subset $A\subseteq\left[k\right]$, we will denote
\[
S_{A}=\bigcup_{j\in A}S_{j}.
\]
 We assume  $\left|S_{i}\right|\ge2$ for all $i\in\left[k\right]$
throughout this subsection.\medskip{}

The base polytope of a matroid $M$ on $S$ with rank function $r$
is the intersection of $\left\{ x\left(S\right)=r\left(S\right)\right\} $
and the following half-spaces:
\[
\ensuremath{\begin{cases}
\left\{ x_{i}\ge0\right\}  & \mbox{ for all }i\in S,\\
\left\{ x\left(A\right)\le r\left(A\right)\right\}  & \mbox{ for all subsets }A\subseteq S.
\end{cases}}
\]
 Here, every vector $1^{A}$ is normal to the face $\mathrm{BP}_{M}\cap\left\{ x\left(A\right)=r\left(A\right)\right\} $
of $\mathrm{BP}_{M}$, and moreover outward-pointing normal to $\mathrm{BP}_{M}$
through the face.\medskip{}

Let $P=\mathrm{tconv}\left(\mathbf{v}_{1},\dots,\mathbf{v}_{k}\right)$
in $\mathbb{R}^{k}/\mathbb{R}\mathbbm1$ be a biconvex polytope. We
may assume  $P$ is maximal. For any vertex $\mathbf{w}$ of $P$,
$\Gamma\left[G_{\mathbf{w}}\right]$ is connected by Lemma \ref{lem:connected},
which has precisely $k-1$ non-degenerate flats of the form $S_{G_{\mathbf{w}}^{+}(i,c)}$.
Then, $\mathrm{BP}_{\Gamma\left[G_{\mathbf{w}}\right]}$ is full-dimensional,
and the vectors $1^{S_{G_{\mathbf{w}}^{+}(i,c)}}$ are outward-pointing
normal to $\mathrm{BP}_{\Gamma\left[G_{\mathbf{w}}\right]}$.\medskip{}

Let $Q$ be an edge of $P$ with $\mathrm{Vert}\left(Q\right)=\left\{ \mathbf{v},\mathbf{w}\right\} $.
Then, $S_{\mathrm{L}^{\mathbf{v}}(Q)}$ and $S_{\mathrm{L}^{\mathbf{w}}(Q)}$
are non-degenerate flats of $\Gamma\left[G_{\mathbf{v}}\right]$ and
$\Gamma\left[G_{\mathbf{w}}\right]$, respectively, with $\left[k\right]=\mathrm{L}^{\mathbf{v}}(Q)\sqcup\mathrm{L}^{\mathbf{w}}(Q)$
and $S=S_{\mathrm{L}^{\mathbf{v}}(Q)}\sqcup S_{\mathrm{L}^{\mathbf{w}}(Q)}$.\medskip{}

Using the notation (\ref{eq:paren}), two matroids $\Gamma\left[G_{\mathbf{v}}\right](S_{\mathrm{L}^{\mathbf{v}}(Q)})$
and $\Gamma\left[G_{\mathbf{w}}\right](S_{\mathrm{L}^{\mathbf{w}}(Q)})$
are \emph{face matroids} of $\Gamma\left[G_{\mathbf{v}}\right]$ and
$\Gamma\left[G_{\mathbf{w}}\right]$, respectively, which are the
same. And we have 
\[
\mathrm{BP}_{\Gamma\left[G_{\mathbf{v}}\right](S_{\mathrm{L}^{\mathbf{v}}(Q)})}=\mathrm{BP}_{\Gamma\left[G_{\mathbf{w}}\right](S_{\mathrm{L}^{\mathbf{w}}(Q)})}=\mathrm{BP}_{\Gamma\left[G_{\mathbf{v}}\right]}\cap\mathrm{BP}_{\Gamma\left[G_{\mathbf{w}}\right]}
\]
 which is the common facet of $\mathrm{BP}_{\Gamma\left[G_{\mathbf{v}}\right]}$
and $\mathrm{BP}_{\Gamma\left[G_{\mathbf{w}}\right]}$. Moreover,
for $V=\left[k\right]$:
\[
\bigcap_{\mathbf{w}\in\mathrm{Vert}\left(P\right)}\mathrm{BP}_{\Gamma\left[G_{\mathbf{w}}\right]}=\mathrm{BP}_{\Gamma\left[V\right]}.
\]
Therefore
\[
\Sigma=\left\{ \mathrm{BP}_{\Gamma\left[G_{\mathbf{w}}\right]}:\mathbf{w}\in\mathrm{Vert}\left(P\right)\right\} 
\]
 is a \emph{matroid tiling}, that is, a \emph{face-fitting} collection
of base polytopes that is connected in codimension $1$. By Corollary
\ref{cor:nondeg-2}, its \emph{support} $\left|\Sigma\right|$ has
no facets that are not contained in the boundary of $\Delta$, which
means that $\left|\Sigma\right|=\Delta$ and $\Sigma$ is a matroid
subdivision of $\Delta$.\medskip{}

Let $\pi:\mathbb{R}^{S}\rightarrow\mathbb{R}^{k}$ be a linear map
defined by 
\[
\mathbf{x}=\sum_{s\in S}x_{s}\mathbf{e}_{s}\mapsto\mathbf{y}=\sum_{i\in\left[k\right]}\left(\sum_{s\in S_{i}}x_{s}\right)\mathbf{e}_{i}.
\]
 Then, the following hold.
\begin{itemize}
\item The hyperplane $\left\{ x\left(S_{i}\right)=1\right\} $ of $\mathbb{R}^{S}$
maps to the hyperplane $\left\{ y_{i}=1\right\} $ of $\mathbb{R}^{k}$.
\item The hyperplane $\left\{ x\left(S\right)=k\right\} $ maps to $\left\{ y\left(V\right)=k\right\} $
for $V=\left[k\right]$.
\item If $\mathrm{Q}=\mathrm{BP}_{\Gamma\left[V\right]}=\mathrm{BP}_{\bigoplus_{i\in\left[k\right]}U_{S_{i}}^{1}}$,
then $\pi\left(\mathrm{Q}\right)=\left\{ \mathbbm1\right\} $.
\item If $\left|S_{i}\right|\ge k$ for all $i\in\left[k\right]$, then
$\pi\left(\Delta_{S}^{k}\right)$ is a regular $\left(k-1\right)$-simplex.
\end{itemize}
For any polytope $\mathrm{P}\subseteq\Delta_{S}^{k}$, let us call
its image $\pi\left(\mathrm{P}\right)$ a \textbf{quotient polytope}.
Denote 
\[
\pi\left(\Sigma\right)=\left\{ \pi\left(\mathrm{P}\right):\mathrm{P}\in\Sigma\right\} 
\]
 which we call a \textbf{quotient tiling} if $\Sigma$ is a tiling,
and a \textbf{quotient subdivision} if $\Sigma$ is a subdivision.\medskip{}

Because it is convenient to have $\pi\left(\Delta_{S}^{k}\right)$
being a regular $\left(k-1\right)$-simplex, let us assume $\left|S_{i}\right|\ge k$
for all $i\in\left[k\right]$.\medskip{}

For a $\mathrm{BP}_{\Gamma\left[G_{\mathbf{w}}\right]}\in\Sigma$,
its facets not contained in the boundary of $\Delta_{S}^{k}$ are
precisely $\mathrm{BP}_{\Gamma\left[G_{\mathbf{w}}(i,c)\right]}$
for the $k-1$ arrows $\left(i,c\right)$ of $G_{\mathbf{w}}$. The
vector $1^{V_{G_{\mathbf{w}}^{+}(i,c)}}$ is outward-pointing normal
to the quotient polytope $\pi\left(\mathrm{BP}_{\Gamma\left[G_{\mathbf{w}}\right]}\right)$.\medskip{}

Now, denote $\mathrm{P}^{\ast}=\mathbbm1-\mathrm{P}$ for a polytope
$\mathrm{P}$, which is an involution of $\mathrm{P}$. Let $\Sigma^{\ast}$
be the collection of those involutions of all members of $\Sigma$:
\[
\Sigma^{\ast}=\left\{ \mathrm{P}^{\ast}\subseteq\Delta_{S}^{\left|S\right|-k}:\mathrm{P}\in\Sigma\right\} 
\]
 which is a polyhedral subdivision of the hypersimplex $\left(\Delta_{S}^{k}\right)^{\ast}=\Delta_{S}^{\left|S\right|-k}$.\smallskip{}

For each $\mathrm{BP}_{\Gamma\left[G_{\mathbf{w}}\right]}\in\Sigma$,
we have $\mathrm{BP}_{\Gamma\left[G_{\mathbf{w}}\right]^{\ast}}=\mathbbm1-\mathrm{BP}_{\Gamma\left[G_{\mathbf{w}}\right]}\in\Sigma^{\ast}$
and vice versa, where $\Gamma\left[G_{\mathbf{w}}\right]^{\ast}$
is the dual matroid of $\Gamma\left[G_{\mathbf{w}}\right]$. So, $\Sigma^{\ast}$
is a matroid subdivision, and the vector $-1^{V_{G_{\mathbf{w}}^{+}(i,c)}}$
is outward-pointing normal to $\pi\left(\mathrm{BP}_{\Gamma\left[G_{\mathbf{w}}\right]^{\ast}}\right)$.\medskip{}

Thus, the biconvex polytope $P$ is the bounded part of a polyhedral
complex that is dual to $\pi\left(\Sigma^{\ast}\right)$.
\begin{rem}
The polyhedral subdivisions $\pi\left(\Sigma\right)$ of $\pi\left(\Delta\right)$
and $\pi\left(\Sigma^{\ast}\right)$ of $\pi\left(\Delta^{\ast}\right)$,
and the matroid subdivisions $\Sigma$ of $\Delta$ and $\Sigma^{\ast}$
of $\Delta^{\ast}$ are all \emph{regular}.
\end{rem}

\subsection{Face-fitting directed bigraphs}

Let $G_{1}$ and $G_{2}$ be two directed bigraphs with the same node
set $V_{G_{1}}=V_{G_{2}}=\left[k\right]$, and fix an underlying partition.
Then, $\mathrm{BP}_{\Gamma\left[V_{G_{1}}\right]}$ is a common face
of $\mathrm{BP}_{\Gamma\left[G_{1}\right]}$ and $\mathrm{BP}_{\Gamma\left[G_{2}\right]}$.\smallskip{}

Suppose that none of $\mathrm{BP}_{\Gamma\left[G_{1}\right]}$ and
$\mathrm{BP}_{\Gamma\left[G_{2}\right]}$ contains the other, and
that $\mathrm{BP}_{\Gamma\left[G_{1}\right]}\cap\mathrm{BP}_{\Gamma\left[G_{2}\right]}$
is a common face of them. Then, there is a maximum common subgraph
$G'$ of $G_{1}$ and $G_{2}$ with $\mathrm{BP}_{\Gamma\left[G_{1}\right]}\cap\mathrm{BP}_{\Gamma\left[G_{2}\right]}=\mathrm{BP}_{\Gamma\left[G'\right]}$.
Moreover, there are subgraphs $G_{1}'$ and $G_{2}'$ of $G_{1}$
and $G_{2}$, respectively, with 
\begin{align}
G_{1}'(e_{1}) & =G_{2}'(e_{2})\nonumber \\
(G_{1}')^{+}(e_{1}) & =(G_{2}')^{-}(e_{2})\label{eq:face-fitting}\\
(G_{1}')^{-}(e_{1}) & =(G_{2}')^{+}(e_{2})\nonumber 
\end{align}
 for some arrows $e_{1}$ and $e_{2}$ of $G_{1}'$ and $G_{2}'$,
respectively.\medskip{}

Conversely, if there is a maximum common subgraph $G'$ of $G_{1}$
and $G_{2}$ with subgraphs $G_{1}'$ and $G_{2}'$, respectively,
satisfying (\ref{eq:face-fitting}), then $\mathrm{BP}_{\Gamma\left[G_{1}\right]}\cap\mathrm{BP}_{\Gamma\left[G_{2}\right]}$
is a common proper face of $\mathrm{BP}_{\Gamma\left[G_{1}\right]}$
and $\mathrm{BP}_{\Gamma\left[G_{2}\right]}$. In this case, we say
that the two directed bigraphs $G_{1}$ and $G_{2}$ are \textbf{face-fitting}.
We say that multiple directed bigraphs with the same node set are
\textbf{face-fitting} if they are pairwise face-fitting.
\begin{example}
The directed bigraphs $G_{1}$ and $G_{2}$ of Figure \ref{fig:bigraphs0}
have parts $\left(\left\{ 1\right\} ,\left\{ 2,3,4\right\} \right)$
and $\left(\left\{ 1,4\right\} ,\left\{ 2,3\right\} \right)$, respectively,
and are face-fitting since $G_{1}(1,4)=G_{2}(3,4)$ with $G_{1}^{+}(1,4)=G_{2}^{-}(3,4)$
and $G_{1}^{-}(1,4)=G_{2}^{+}(3,4)$. Then
\[
\mathrm{BP}_{\Gamma\left[G_{1}\right]}\cap\mathrm{BP}_{\Gamma\left[G_{2}\right]}=\mathrm{BP}_{\Gamma\left[G_{1}(1,4)\right]}=\mathrm{BP}_{\Gamma\left[G_{2}(3,4)\right]}
\]
 is a codimension-$1$ common face of $\mathrm{BP}_{\Gamma\left[G_{1}\right]}$
and $\mathrm{BP}_{\Gamma\left[G_{2}\right]}$, and $\left\{ \mathrm{BP}_{\Gamma\left[G_{1}\right]},\mathrm{BP}_{\Gamma\left[G_{2}\right]}\right\} $
is a matroid tiling.
\begin{figure}[th]
\noindent \centering{}\noindent \begin{center}
\begin{tikzpicture}[blowup line/.style={line width=2pt,grey},font=\scriptsize]

\matrix[column sep=1.5cm, row sep=0.6cm]{

\begin{scope}[line cap=round,scale=1,decoration={markings,mark=at position 0.65 with {\arrow{Stealth}}}]

 \path (-1,1) coordinate(b) node[above]{$2$};
 \path (0,1) coordinate(c) node[above]{$3$};
 \path (1,1) coordinate(d) node[above]{$4$};
 \path (0,0) coordinate(a) node[below]{$1$};

 \foreach \x in {b,c,d}{
          \draw [thick,postaction={decorate}] (a)--(\x);}

 \foreach \x in {a,b,c,d}{
          \fill [black](\x) circle (2.2pt);}
 \path (0,1.6) node{$G_{1}$};
\end{scope}

&

\begin{scope}[line cap=round,scale=1,decoration={markings,mark=at position 0.65 with {\arrow{Stealth}}}]

 \path (-1,1) coordinate(b) node[above]{$2$};
 \path (0,1) coordinate(c) node[above]{$3$};
 \path (1,1) coordinate(d) node[above]{$4$};
 \path (0,0) coordinate(a) node[below]{$1$};

 \foreach \x in {b,c}{
          \draw [thick,postaction={decorate}] (a)--(\x);}

 \foreach \x in {a,b,c,d}{
          \fill [black](\x) circle (2.2pt);}
 \path (0,1.6) node{$G_{1}(1,4)$};
\end{scope}

&

\begin{scope}[line cap=round,xshift=0.5cm,scale=1,decoration={markings,mark=at position 0.65 with {\arrow{Stealth}}}]

 \path (-1,1) coordinate(b) node[above]{$2$};
 \path (0,1) coordinate(c) node[above]{$3$};
 \path (0,0) coordinate(a) node[below]{$1$};

 \foreach \x in {b,c}{
          \draw [thick,postaction={decorate}] (a)--(\x);}

 \foreach \x in {a,b,c}{
          \fill [black](\x) circle (2.2pt);}
\end{scope}
\path (0,1.6) node{$G_{1}^{+}(1,4)$};

&

\begin{scope}[line cap=round,xshift=-1cm,yshift=-.5cm,scale=1,decoration={markings,mark=at position 0.65 with {\arrow{Stealth}}}]

 \path (1,1) coordinate(d) node[above]{$4$};
 \foreach \x in {d}{
          \fill [black](\x) circle (2.2pt);}
\end{scope}
\path (0,1.6) node{$G_{1}^{-}(1,4)$};
\\\hline
\\

\begin{scope}[line cap=round,scale=1,decoration={markings,mark=at position 0.65 with {\arrow{Stealth}}}]

 \path (-.5,1) coordinate(b) node[above]{$2$};
 \path (0.5,1) coordinate(c) node[above]{$3$};
 \path (-.5,0) coordinate(d) node[below]{$4$};
 \path (0.5,0) coordinate(a) node[below]{$1$};

 \foreach \x in {b,c}{
          \draw [thick,postaction={decorate}] (a)--(\x);}

 \draw [thick,postaction={decorate}] (d)--(b);

 \foreach \x in {a,b,c,d}{
          \fill [black](\x) circle (2.2pt);}
 \path (0,1.6) node{$G_{2}$};
\end{scope}

&

\begin{scope}[line cap=round,scale=1,decoration={markings,mark=at position 0.65 with {\arrow{Stealth}}}]

 \path (-.5,1) coordinate(b) node[above]{$2$};
 \path (0.5,1) coordinate(c) node[above]{$3$};
 \path (-.5,0) coordinate(d) node[below]{$4$};
 \path (0.5,0) coordinate(a) node[below]{$1$};

 \foreach \x in {b,c}{
          \draw [thick,postaction={decorate}] (a)--(\x);}

 \foreach \x in {a,b,c,d}{
          \fill [black](\x) circle (2.2pt);}
 \path (0,1.6) node{$G_{2}(4,2)$};
\end{scope}

&

\begin{scope}[line cap=round,xshift=.5cm,yshift=.5cm,scale=1,decoration={markings,mark=at position 0.65 with {\arrow{Stealth}}}]

 \path (-.5,0) coordinate(d) node[below]{$4$};

 \foreach \x in {d}{
          \fill [black](\x) circle (2.2pt);}
\end{scope}
\path (0,1.6) node{$G_{2}^{+}(4,2)$};

&

\begin{scope}[line cap=round,scale=1,decoration={markings,mark=at position 0.65 with {\arrow{Stealth}}}]

 \path (-.5,1) coordinate(b) node[above]{$2$};
 \path (0.5,1) coordinate(c) node[above]{$3$};
 \path (0.5,0) coordinate(a) node[below]{$1$};

 \foreach \x in {b,c}{
          \draw [thick,postaction={decorate}] (a)--(\x);}

 \foreach \x in {a,b,c}{
          \fill [black](\x) circle (2.2pt);}
 \path (0,1.6) node{$G_{2}^{-}(4,2)$};
\end{scope}

\\
};

\end{tikzpicture}
\par\end{center}\caption{\label{fig:bigraphs0}Face-Fitting Directed Bigraphs, I}
\end{figure}
\end{example}

\begin{example}
The $6$ directed bigraphs $G_{1},\dots,G_{6}$ of Figure \ref{fig:bigraphs1}
are face-fitting, and $\Sigma=\left\{ \mathrm{BP}_{\Gamma\left[G_{l}\right]}:l\in\left[6\right]\right\} $
is a matroid tiling with $5$ \emph{connecting} facets, i.e. $5$
common facets of two polytopes of $\Sigma$, whose matroids are:
\begin{itemize}
\item $\Gamma\left[G_{2}(4,2)\right]=\Gamma\left[G_{1}(1,4)\right]$,
\item $\Gamma\left[G_{2}(1,3)\right]=\Gamma\left[G_{5}(3,2)\right]$,
\item $\Gamma\left[G_{3}(1,3)\right]=\Gamma\left[G_{4}(4,1)\right]$,
\item $\Gamma\left[G_{3}(4,2)\right]=\Gamma\left[G_{6}(2,3)\right]$,
\item $\Gamma\left[G_{2}(1,2)\right]=\Gamma\left[G_{3}(4,3)\right]$.
\end{itemize}
\begin{figure}[th]
\noindent \centering{}\noindent \begin{center}
\begin{tikzpicture}[blowup line/.style={line width=2pt,grey},font=\scriptsize]

\matrix[column sep=0.25cm, row sep=0.25cm]{

\begin{scope}[line cap=round,scale=1,decoration={markings,mark=at position 0.65 with {\arrow{Stealth}}}]

 \path (-1,1) coordinate(b) node[above]{$2$};
 \path (0,1) coordinate(c) node[above]{$3$};
 \path (1,1) coordinate(d) node[above]{$4$};
 \path (0,0) coordinate(a) node[below]{$1$};

 \foreach \x in {b,c,d}{
          \draw [thick,postaction={decorate}] (a)--(\x);}

 \foreach \x in {a,b,c,d}{
          \fill [black](\x) circle (2.2pt);}
 \path (0,1.6) node{$G_{1}$};
\end{scope}

&

 \path (0,.8) node{$_{G_{1}(1,4)=G_{2}(4,2)}$};

\draw (-1,0.5)--(1,0.5);

\begin{scope}[line cap=round,xshift=-0.25cm,yshift=-0.4cm,scale=0.5,decoration={markings,mark=at position 0.75 with {\arrow{Stealth}}}]

 \path (-.5,1) coordinate(b) node[above=-1pt]{$_2$};
 \path (0.5,1) coordinate(c) node[above=-1pt]{$_3$};
 \path (1.25,0.5) coordinate(d) node[right=-1pt]{$_4$};
 \path (0.5,0) coordinate(a) node[below=-1pt]{$_1$};

 \foreach \x in {b,c}{
          \draw [thick,postaction={decorate}] (a)--(\x);}

 \foreach \x in {a,b,c,d}{
          \fill [black](\x) circle (2.2pt);}
\end{scope}

&

\begin{scope}[line cap=round,scale=1,decoration={markings,mark=at position 0.6 with {\arrow{Stealth}}}]

 \path (-.5,1) coordinate(b) node[above]{$2$};
 \path (0.5,1) coordinate(c) node[above]{$3$};
 \path (-.5,0) coordinate(d) node[below]{$4$};
 \path (0.5,0) coordinate(a) node[below]{$1$};

 \foreach \x in {b,c}{
          \draw [thick,postaction={decorate}] (a)--(\x);}

 \draw [thick,postaction={decorate}] (d)--(b);

 \foreach \x in {a,b,c,d}{
          \fill [black](\x) circle (2.2pt);}
 \path (0,1.6) node{$G_{2}$};
\end{scope}

&

 \path (0,.8) node{$_{G_{2}(1,3)=G_{5}(3,2)}$};

\draw (-1,0.5)--(1,0.5);

\begin{scope}[line cap=round,xshift=-0.25cm,yshift=-0.4cm,decoration={markings,mark=at position 0.65 with {\arrow{Stealth}}},scale=0.5]
 \path (-.5,1) coordinate(b) node[above=-1pt]{$_2$};
 \path (1.25,0.5) coordinate(c) node[right=-1pt]{$_3$};
 \path (-.5,0) coordinate(d) node[below=-1pt]{$_4$};
 \path (0.5,0) coordinate(a) node[below=-1pt]{$_1$};

 \foreach \x in {a,d}{
          \draw [thick,postaction={decorate}] (\x)--(b);}

 \foreach \x in {a,b,c,d}{
          \fill [black](\x) circle (2.2pt);}
\end{scope}

&

\begin{scope}[line cap=round,scale=1,decoration={markings,mark=at position 0.6 with {\arrow{Stealth}}}]

 \path (-1,0) coordinate(b) node[below]{$3$};
 \path (0,0) coordinate(c) node[below]{$4$};
 \path (1,0) coordinate(d) node[below]{$1$};
 \path (0,1) coordinate(a) node[above]{$2$};

 \foreach \x in {b,c,d}{
          \draw [thick,postaction={decorate}] (\x)--(a);}

 \foreach \x in {a,b,c,d}{
          \fill [black](\x) circle (2.2pt);}
 \path (0,1.6) node{$G_{5}$};
\end{scope}

\\

&
&

 \path (-.7,.5) node{$\begin{array}{c}_{G_{2}(1,2)}\\{\shortparallel}\\_{G_{3}(4,3)}\end{array}$};

\draw (0,-0.5)--(0,1.5);

\begin{scope}[line cap=round,xshift=.7cm,yshift=0.25cm,decoration={markings,mark=at position 0.75 with {\arrow{Stealth}}},scale=0.5]

 \path (-.5,1) coordinate(b) node[above=-1pt]{$_2$};
 \path (0.5,1) coordinate(c) node[above=-1pt]{$_3$};
 \path (-.5,0) coordinate(d) node[below=-1pt]{$_4$};
 \path (0.5,0) coordinate(a) node[below=-1pt]{$_1$};

 \draw [thick,postaction={decorate}] (a)--(c);
 \draw [thick,postaction={decorate}] (d)--(b);

 \foreach \x in {a,b,c,d}{
          \fill [black](\x) circle (2.2pt);}
\end{scope}

&
&

\\

\begin{scope}[line cap=round,scale=1,decoration={markings,mark=at position 0.65 with {\arrow{Stealth}}}]

 \path (-1,1) coordinate(b) node[above]{$1$};
 \path (0,1) coordinate(c) node[above]{$2$};
 \path (1,1) coordinate(d) node[above]{$3$};
 \path (0,0) coordinate(a) node[below]{$4$};

 \foreach \x in {b,c,d}{
          \draw [thick,postaction={decorate}] (a)--(\x);}

 \foreach \x in {a,b,c,d}{
          \fill [black](\x) circle (2.2pt);}
 \path (0,-.6) node{$G_{4}$};
\end{scope}

&

 \path (0,0.2) node{$_{G_{4}(4,1)=G_{3}(1,3)}$};

\draw (-1,0.5)--(1,0.5);

\begin{scope}[line cap=round,xshift=-0.25cm,yshift=0.9cm,decoration={markings,mark=at position 0.75 with {\arrow{Stealth}}},scale=0.5]

 \path (-.5,1) coordinate(b) node[above=-1pt]{$_2$};
 \path (0.5,1) coordinate(c) node[above=-1pt]{$_3$};
 \path (-.5,0) coordinate(d) node[below=-1pt]{$_4$};
 \path (1.25,0.5) coordinate(a) node[right=-1pt]{$_1$};

 \draw [thick,postaction={decorate}] (d)--(b);
 \draw [thick,postaction={decorate}] (d)--(c);

 \foreach \x in {a,b,c,d}{
          \fill [black](\x) circle (2.2pt);}
\end{scope}

&

\begin{scope}[line cap=round,scale=1,decoration={markings,mark=at position 0.6 with {\arrow{Stealth}}}]

 \path (-.5,1) coordinate(b) node[above]{$2$};
 \path (0.5,1) coordinate(c) node[above]{$3$};
 \path (-.5,0) coordinate(d) node[below]{$4$};
 \path (0.5,0) coordinate(a) node[below]{$1$};

 \foreach \x in {b,c}{
          \draw [thick,postaction={decorate}] (d)--(\x);}

 \draw [thick,postaction={decorate}] (a)--(c);

 \foreach \x in {a,b,c,d}{
          \fill [black](\x) circle (2.2pt);}
 \path (0,-.6) node{$G_{3}$};
\end{scope}

&

 \path (0,.2) node{$_{G_{3}(4,2)=G_{6}(2,3)}$};

\draw (-1,0.5)--(1,0.5);

\begin{scope}[line cap=round,xshift=-0.25cm,yshift=0.9cm,decoration={markings,mark=at position 0.65 with {\arrow{Stealth}}},scale=0.5]
 \path (1.25,0.5) coordinate(b) node[right=-1pt]{$_2$};
 \path (0.5,1) coordinate(c) node[above=-1pt]{$_3$};
 \path (-.5,0) coordinate(d) node[below=-1pt]{$_4$};
 \path (0.5,0) coordinate(a) node[below=-1pt]{$_1$};

 \draw [thick,postaction={decorate}] (a)--(c);
 \draw [thick,postaction={decorate}] (d)--(c);

 \foreach \x in {a,b,c,d}{
          \fill [black](\x) circle (2.2pt);}
\end{scope}

&

\begin{scope}[line cap=round,scale=1,decoration={markings,mark=at position 0.6 with {\arrow{Stealth}}}]

 \path (-1,0) coordinate(b) node[below]{$4$};
 \path (0,0) coordinate(c) node[below]{$1$};
 \path (1,0) coordinate(d) node[below]{$2$};
 \path (0,1) coordinate(a) node[above]{$3$};

 \foreach \x in {b,c,d}{
          \draw [thick,postaction={decorate}] (\x)--(a);}

 \foreach \x in {a,b,c,d}{
          \fill [black](\x) circle (2.2pt);}
 \path (0,-.6) node{$G_{6}$};
\end{scope}

\\
};

\end{tikzpicture}
\par\end{center}\caption{\label{fig:bigraphs1}Face-Fitting Directed Bigraphs, II}
\end{figure}
\end{example}

\begin{example}
Replace $G_{2}$ and $G_{3}$ of Figure \ref{fig:bigraphs1} with
$G_{2}'$ and $G_{3}'$ of Figure \ref{fig:bigraphs2}, then $G_{1},G_{4},G_{5},G_{6},G_{2}',G_{3}'$
are face-fitting and the corresponding matroid tiling with $5$ connecting
facets whose matroids are:
\begin{itemize}
\item $\Gamma\left[G'_{2}(4,3)\right]=\Gamma\left[G_{1}(1,4)\right]$,
\item $\Gamma\left[G'_{2}(1,2)\right]=\Gamma\left[G_{6}(2,3)\right]$,
\item $\Gamma\left[G'_{3}(1,2)\right]=\Gamma\left[G_{4}(4,1)\right]$,
\item $\Gamma\left[G'_{3}(4,3)\right]=\Gamma\left[G_{5}(3,2)\right]$,
\item $\Gamma\left[G'_{2}(1,3)\right]=\Gamma\left[G'_{3}(4,2)\right]$.
\end{itemize}
\begin{figure}[th]
\noindent \centering{}\noindent \begin{center}
\begin{tikzpicture}[blowup line/.style={line width=2pt,grey},font=\scriptsize]

\matrix[column sep=0.25cm, row sep=0.25cm]{

\begin{scope}[line cap=round,scale=1,decoration={markings,mark=at position 0.65 with {\arrow{Stealth}}}]

 \path (-1,1) coordinate(b) node[above]{$4$};
 \path (0,1) coordinate(c) node[above]{$2$};
 \path (1,1) coordinate(d) node[above]{$3$};
 \path (0,0) coordinate(a) node[below]{$1$};

 \foreach \x in {b,c,d}{
          \draw [thick,postaction={decorate}] (a)--(\x);}

 \foreach \x in {a,b,c,d}{
          \fill [black](\x) circle (2.2pt);}
 \path (0,1.6) node{$G_{1}$};
\end{scope}

&

 \path (0,.8) node{$_{G_{1}(1,4)=G'_{2}(4,3)}$};

\draw (-1,0.5)--(1,0.5);

\begin{scope}[line cap=round,xshift=-0.25cm,yshift=-0.4cm,scale=0.5,decoration={markings,mark=at position 0.75 with {\arrow{Stealth}}}]

 \path (-.5,1) coordinate(b) node[above=-1pt]{$_2$};
 \path (0.5,1) coordinate(c) node[above=-1pt]{$_3$};
 \path (-.5,0) coordinate(d) node[below=-1pt]{$_1$};
 \path (1.25,0.5) coordinate(a) node[right=-1pt]{$_4$};

 \draw [thick,postaction={decorate}] (d)--(b);
 \draw [thick,postaction={decorate}] (d)--(c);

 \foreach \x in {a,b,c,d}{
          \fill [black](\x) circle (2.2pt);}
\end{scope}

&

\begin{scope}[line cap=round,scale=1,decoration={markings,mark=at position 0.6 with {\arrow{Stealth}}}]

 \path (-.5,1) coordinate(b) node[above]{$2$};
 \path (0.5,1) coordinate(c) node[above]{$3$};
 \path (-.5,0) coordinate(d) node[below]{$1$};
 \path (0.5,0) coordinate(a) node[below]{$4$};

 \foreach \x in {b,c}{
          \draw [thick,postaction={decorate}] (d)--(\x);}

 \draw [thick,postaction={decorate}] (a)--(c);

 \foreach \x in {a,b,c,d}{
          \fill [black](\x) circle (2.2pt);}
 \path (0,1.6) node{$G'_{2}$};
\end{scope}

&

 \path (0,.8) node{$_{G'_{2}(1,2)=G_{6}(2,3)}$};

\draw (-1,0.5)--(1,0.5);

\begin{scope}[line cap=round,xshift=-0.25cm,yshift=-0.4cm,scale=0.5,decoration={markings,mark=at position 0.65 with {\arrow{Stealth}}}]

 \path (1.25,0.5) coordinate(b) node[right=-1pt]{$_2$};
 \path (0.5,1) coordinate(c) node[above=-1pt]{$_3$};
 \path (-.5,0) coordinate(d) node[below=-1pt]{$_1$};
 \path (0.5,0) coordinate(a) node[below=-1pt]{$_4$};

 \draw [thick,postaction={decorate}] (a)--(c);
 \draw [thick,postaction={decorate}] (d)--(c);

 \foreach \x in {a,b,c,d}{
          \fill [black](\x) circle (2.2pt);}
\end{scope}

&

\begin{scope}[line cap=round,scale=1,decoration={markings,mark=at position 0.6 with {\arrow{Stealth}}}]

 \path (-1,0) coordinate(b) node[below]{$1$};
 \path (0,0) coordinate(c) node[below]{$4$};
 \path (1,0) coordinate(d) node[below]{$2$};
 \path (0,1) coordinate(a) node[above]{$3$};

 \foreach \x in {b,c,d}{
          \draw [thick,postaction={decorate}] (\x)--(a);}

 \foreach \x in {a,b,c,d}{
          \fill [black](\x) circle (2.2pt);}
 \path (0,1.6) node{$G_{6}$};
\end{scope}

\\

&
&

 \path (-.7,.5) node{$\begin{array}{c}_{G'_{2}(1,3)}\\{\shortparallel}\\_{G'_{3}(4,2)}\end{array}$};

\draw (0,-0.5)--(0,1.5);

\begin{scope}[line cap=round,xshift=.7cm,yshift=0.25cm,scale=0.5,decoration={markings,mark=at position 0.75 with {\arrow{Stealth}}}]

 \path (-.5,1) coordinate(b) node[above=-1pt]{$_2$};
 \path (0.5,1) coordinate(c) node[above=-1pt]{$_3$};
 \path (-.5,0) coordinate(d) node[below=-1pt]{$_1$};
 \path (0.5,0) coordinate(a) node[below=-1pt]{$_4$};

 \draw [thick,postaction={decorate}] (a)--(c);
 \draw [thick,postaction={decorate}] (d)--(b);

 \foreach \x in {a,b,c,d}{
          \fill [black](\x) circle (2.2pt);}
\end{scope}

&
&

\\

\begin{scope}[line cap=round,scale=1,decoration={markings,mark=at position 0.65 with {\arrow{Stealth}}}]

 \path (-1,1) coordinate(b) node[above]{$2$};
 \path (0,1) coordinate(c) node[above]{$3$};
 \path (1,1) coordinate(d) node[above]{$1$};
 \path (0,0) coordinate(a) node[below]{$4$};

 \foreach \x in {b,c,d}{
          \draw [thick,postaction={decorate}] (a)--(\x);}

 \foreach \x in {a,b,c,d}{
          \fill [black](\x) circle (2.2pt);}
 \path (0,-.6) node{$G_{4}$};
\end{scope}

&

 \path (0,.2) node{$_{G_{4}(4,1)=G'_{3}(1,2)}$};

\draw (-1,0.5)--(1,0.5);

\begin{scope}[line cap=round,xshift=-0.25cm,yshift=0.9cm,scale=0.5,decoration={markings,mark=at position 0.75 with {\arrow{Stealth}}}]

 \path (-.5,1) coordinate(b) node[above=-1pt]{$_2$};
 \path (0.5,1) coordinate(c) node[above=-1pt]{$_3$};
 \path (1.25,0.5) coordinate(d) node[right=-1pt]{$_1$};
 \path (0.5,0) coordinate(a) node[below=-1pt]{$_4$};

 \foreach \x in {b,c}{
          \draw [thick,postaction={decorate}] (a)--(\x);}

 \foreach \x in {a,b,c,d}{
          \fill [black](\x) circle (2.2pt);}
\end{scope}

&

\begin{scope}[line cap=round,scale=1,decoration={markings,mark=at position 0.6 with {\arrow{Stealth}}}]

 \path (-.5,1) coordinate(b) node[above]{$2$};
 \path (0.5,1) coordinate(c) node[above]{$3$};
 \path (-.5,0) coordinate(d) node[below]{$1$};
 \path (0.5,0) coordinate(a) node[below]{$4$};

 \foreach \x in {b,c}{
          \draw [thick,postaction={decorate}] (a)--(\x);}

 \draw [thick,postaction={decorate}] (d)--(b);

 \foreach \x in {a,b,c,d}{
          \fill [black](\x) circle (2.2pt);}
 \path (0,-.6) node{$G'_{3}$};
\end{scope}

&

 \path (0,.2) node{$_{G'_{3}(4,3)=G_{5}(3,2)}$};

\draw (-1,0.5)--(1,0.5);

\begin{scope}[line cap=round,xshift=-0.25cm,yshift=0.9cm,scale=0.5,decoration={markings,mark=at position 0.65 with {\arrow{Stealth}}}]

 \path (-.5,1) coordinate(b) node[above=-1pt]{$_2$};
 \path (1.25,0.5) coordinate(c) node[right=-1pt]{$_3$};
 \path (-.5,0) coordinate(d) node[below=-1pt]{$_1$};
 \path (0.5,0) coordinate(a) node[below=-1pt]{$_4$};

 \foreach \x in {a,d}{
          \draw [thick,postaction={decorate}] (\x)--(b);}

 \foreach \x in {a,b,c,d}{
          \fill [black](\x) circle (2.2pt);}
\end{scope}

&

\begin{scope}[line cap=round,scale=1,decoration={markings,mark=at position 0.6 with {\arrow{Stealth}}}]

 \path (-1,0) coordinate(b) node[below]{$3$};
 \path (0,0) coordinate(c) node[below]{$1$};
 \path (1,0) coordinate(d) node[below]{$4$};
 \path (0,1) coordinate(a) node[above]{$2$};

 \foreach \x in {b,c,d}{
          \draw [thick,postaction={decorate}] (\x)--(a);}

 \foreach \x in {a,b,c,d}{
          \fill [black](\x) circle (2.2pt);}
 \path (0,-.6) node{$G_{5}$};
\end{scope}

\\
};

\end{tikzpicture}
\par\end{center}\caption{\label{fig:bigraphs2}Face-Fitting Directed Bigraphs, III}
\end{figure}
\end{example}

\begin{rem}
$\left\{ G_{2},G_{3}\right\} $ and $\left\{ G_{2}',G_{3}'\right\} $
are two unique pairs of directed bigraphs that can be added to $\left\{ G_{1},G_{4},G_{5},G_{6}\right\} $
in order to extend $\left\{ \mathrm{BP}_{\Gamma[G_{l}]}:l=1,4,5,6\right\} $
to a matroid tiling.
\end{rem}

\section{Biconvex Polytope as Cell of Tropical Linear Space}

In this short section, we prove a biconvex polytope arises as a cell
of a tropical linear space. Before doing so, we give a brief review
of the tropical linear spaces. Readers are referred to \cite{MS15}
for more.\medskip{}

Let $M$ be a rank-$k$ connected matroid on $S$ with $n=\left|S\right|$,
and $\mathcal{B}$ its base collection. The \textbf{Dressian}\footnote{This is named after Andreas Dress due to his original work on ``valuated
matroids'' with Walter Wenzel, \cite{Dress}. One might think this
should be called the ``tropical Grassmannian''. However, ``being
generated'' for ideals is not transferred via tropicalization, and
the tropical Grassmannian is defined as the intersection of \emph{all}
tropical hypersurfaces coming from tropicalized elements of the Plücker
ideal.} $\mathrm{Dr}_{M}$ of $M$ is the intersection of the tropical hypersurfaces
in $\mathbb{R}^{\mathcal{B}}/\mathbb{R}\mathbbm1$ defined by the\textbf{
tropicalized Plücker relations }for $M$. Here a tropicalized Plücker
relation for $M$ is a tropical polynomial 
\[
\op_{j\in\tau'}\left.z_{\sigma\cup\left\{ j\right\} }\odot z_{\tau-\left\{ j\right\} }\right.
\]
for a coordinate vector $\mathbf{z}=\left(z_{B}\right)_{B\in\mathcal{B}}$
in $\mathbb{R}^{\mathcal{B}}/\mathbb{R}\mathbbm1$, with the following
properties: 
\begin{itemize}
\item $\sigma$ is an independent set of $M$ of size $k-1$,
\item $\tau$ is a rank-$k$ subset of $S$ of size $k+1$ with $\sigma\nsubseteq\tau$,
\item $\tau'$ is the set of $j$ in $\tau$ such that both $\sigma\cup\left\{ j\right\} $
and $\tau-\left\{ j\right\} $ are bases of $M$.
\end{itemize}
Fix any point $\mathbf{z}$ in $\mathrm{Dr}_{M}\subset\mathbb{R}^{\mathcal{B}}/\mathbb{R}\mathbbm1$.
For any rank-$k$ subset $\tau$ of $S$ of size $k+1$ with $\tau'=\left\{ j\in\tau:\tau-\left\{ j\right\} \in\mathcal{B}\right\} $,
we denote by $L_{\tau}(\mathbf{z})$ the tropical hyperplane in $\mathbb{R}^{S}/\mathbb{R}\mathbbm1$
that is defined by a tropical polynomial 
\[
\op_{j\in\tau'}\left.z_{\tau-\left\{ j\right\} }\odot x_{j}\right.
\]
 for a coordinate vector $\mathbf{x}=\left(x_{j}\right)_{j\in S}\in\mathbb{R}^{S}/\mathbb{R}\mathbbm1$.
Now, define 
\[
L_{\mathbf{z}}:=\bigcap_{\tau}L_{\tau}\left(\mathbf{z}\right)
\]
 which is a $\left(k-1\right)$-dimensional balanced contractible
polyhedral complex in $\mathbb{R}^{S}/\mathbb{R}\mathbbm1$, and called
a \textbf{tropical linear space}.\medskip{}

Every $\mathbf{z}\in\mathbb{R}^{\mathcal{B}}/\mathbb{R}\mathbbm1$
induces a \emph{regular} polyhedral subdivision of the base polytope
$\mathrm{BP}_{M}$. Furthermore:
\begin{prop}
A point $\mathbf{z}$ of $\mathbb{R}^{\mathcal{B}}/\mathbb{R}\mathbbm1$
is contained in $\mathrm{Dr}_{M}$ if and only if the regular subdivision
of $\mathrm{BP}_{M}$ that $\mathbf{z}$ induces is a matroid subdivision.
\end{prop}

Now, let $P=\mathrm{tconv}\left(\mathbf{v}_{1},\dots,\mathbf{v}_{k}\right)\subset\mathbb{R}^{k}/\mathbb{R}\mathbbm1$
be any maximal biconvex polytope. Then, it immediately follows that
for the regular matroid subdivision $\Sigma^{\ast}$ of $\Delta_{S}^{\left|S\right|-k}$
that is constructed in Subsection \ref{subsec:MAIN}, $P$ is tropically
isomorphic to the bounded part of a tropical linear space that is
a polyhedral complex dual to $\Sigma^{\ast}$.\medskip{}

The above statement holds for any biconvex polytope because every
non-maximal biconvex polytope is obtained as a tropical degeneration
of a maximal one.

\section{Subdividing Hypersimplices: Rank-$4$ Case}

We employ our theory to manually conduct all the computations with
pen and paper without resorting to computers, which reflects the power
of our theory. Fix an underlying partition $S=\bigcup_{i\in\left[k\right]}S_{i}$
with $\left|S_{i}\right|\ge k$ for all $i\in\left[k\right]$. We
first borrow necessary tools from \cite{j-hope}, and then work out
the rank-$4$ case.
\begin{lem}
\label{lem:Cutting-1}Let $S$ be a finite set and $k$ an integer
with $1\le k<\left|S\right|$.
\begin{enumerate}
\item \cite[Lemma 4.9]{j-hope} Let $F\subset S$ be a proper subset of
size $\ge2$ and $\rho$ a positive integer with 
\[
k+\left|F\right|-\left|S\right|<\rho<\min\left\{ k,\left|F\right|\right\} .
\]
Then, $\mathrm{P}_{1}=\Delta_{S}^{k}\cap\left\{ x\left(F\right)\le\rho\right\} $
and $\mathrm{P}_{2}=\Delta_{S}^{k}\cap\left\{ x\left(S-F\right)\le k-\rho\right\} $
are full-dimensional base polytopes that are face-fitting, which form
a matroid subdivision of $\Delta_{S}^{k}$.
\item \cite[Lemma 4.9]{j-hope} Let $\mathrm{P}_{1}=\mathrm{BP}_{M_{1}}$
and $\mathrm{P}_{2}=\mathrm{BP}_{M_{2}}$, then $F$ and $S-F$ are
the unique non-degenerate flats of size $\ge2$ of $M_{1}$ and $M_{2}$,
respectively, with ranks $\rho$ and $k-\rho$.
\item \label{enu:subdividing}\cite[Corollary 4.8]{j-hope} For a partition
$S=\bigcup_{j\in\left[k\right]}S_{j}$, cutting $\Delta_{S}^{k}$
with all $k$ hyperplanes of the form $\left\{ x\left(S_{j}\right)=1\right\} $
produces a matroid subdivision.
\end{enumerate}
\end{lem}

\begin{example}
\label{exa:4-subdivision}Let $k=4$. By cutting $\Delta_{S}^{4}$
with all $4$ hyperplanes $\left\{ x\left(S_{i}\right)=1\right\} $
we get a matroid subdivision $\Sigma$. See Figure \ref{fig:Splits}
for the quotient subdivision $\pi\left(\Sigma\right)$. We have $\Gamma\left[V\right]=\bigoplus_{i\in\left[4\right]}U_{S_{i}}^{1}$
for $V=\left[4\right]$, and 
\[
\mathrm{Q}=\mathrm{BP}_{\bigoplus_{i\in\left[4\right]}U_{S_{i}}^{1}}
\]
 is the maximum common face of all members of $\Sigma$ with $\pi\left(Q\right)=\left(1,1,1,1\right)$.
See Figure \ref{fig:Blocks0} for the typical $3$ (quotient) polytopes
of $\Sigma$ and the corresponding directed graphs where the graphs
of Figure \ref{fig:Blocks0} are from Figure \ref{fig:bigraphs1}.
\begin{figure}[th]
\noindent \centering{}\noindent \begin{center}
\def\sizea{0.45}

\begin{tikzpicture}[blowup line/.style={line width=2pt,grey},font=\scriptsize]

\begin{scope}[line cap=round,rotate=0,scale=\sizea,xshift=0cm,yshift=0cm]

\pgfsetxvec{\pgfpointxyz{1.5}{0}{0.5}}
\pgfsetzvec{\pgfpointxyz{-0.2}{0}{1.5}}
\pgfsetyvec{\pgfpointxyz{0}{1.25}{0}}

\fill [nearly opaque,red]
 (0,2.449,0)--++(-1.5,0,-2.598)--++(0,-2.449,-1.732)--++(1.5,0,2.598)--cycle;
\fill [nearly opaque,blue]
 (0,2.449,0)++(-3,0,0)--++(3,0,0)--++(1.5,0,-2.598)--++(-1.5,0,-2.598)--cycle;
\fill [nearly opaque,red]
 (0,2.449,0)--++(-1.5,0,-2.598)--++(0,2.449,1.732)--cycle;
\fill [nearly opaque,green]
 (0,2.449,0)++(1.5,2.449,-0.866)--++(0,-2*2.449,-2*1.732)--++(-2*1.5,0,2*2.598)--cycle;
\fill [nearly opaque,yellow]
 (0,2.449,0)++(1.5,2.449,-0.866)--++(-1.5,2.449,-0.866)--++(-4.5,-3*2.449,3*0.866)--++(3,0,0)--cycle;
\fill [nearly opaque,red]
 (0,2.449,0)--++(-1.5,2.449,-0.866)--++(0,2.449,1.732)--++(1.5,-2.449,0.866)--cycle;
\fill [nearly opaque,blue]
 (0,2.449,0)--++(-1.5,0,2.598)--++(-3,0,0)--++(1.5,0,-2.598)--cycle;
\fill [nearly opaque,green]
 (0,2.449,0)++(1.5,2.449,-0.866)--++(0,2.449,1.732)--++(-3*1.5,-3*2.449,3*0.866)--++(1.5,0,-2.598)--cycle;
\fill [nearly opaque,blue]
 (0,2.449,0)--++(3,0,0)--++(-1.5,0,-2.598)--cycle;
\fill [nearly opaque,red]
 (0,2.449,0)--++(0,-2.449,-1.732)--++(1.5,0,2.598)--cycle;
\fill [nearly opaque,yellow]
 (0,2.449,0)++(-1.5,-2.449,0.866)--++(6,0,0)--++(-3,2*2.449,-2*0.866)--cycle;
\fill [nearly opaque,red]
 (0,2.449,0)--++(1.5,0,2.598)--++(1.5,-2.449,0.866)--++(-1.5,0,-2.598)--cycle;
\fill [nearly opaque,blue]
 (0,2.449,0)--++(-1.5,0,2.598)--++(6,0,0)--++(-1.5,0,-2.598)--cycle;
\fill [nearly opaque,red]
 (0,2.449,0)--++(1.5,0,2.598)--++(-1.5,2.449,-0.866)--cycle;

\draw [line join=round,line cap=round,line width=0.15pt] (0,0,-2*3.464)--++(-4*1.5,0,4*2.598)
               (0,0,-2*3.464)++(1.5,0,2.598)--++(-3*1.5,0,3*2.598)
               (0,0,-2*3.464)++(0,2.449,1.732)--++(-3*1.5,0,3*2.598);

\draw [line join=round,line cap=round,line width=0.15pt] (0,0,-2*3.464)++(-1.5,0,2.598)
--++(0,3*2.449,3*1.732)
               (0,0,-2*3.464)++(-1.5,0,2.598)--++(3*1.5,0,3*2.598);

\draw [line join=round,line cap=round,line width=0.15pt] (0,0,-2*3.464)++(-3*1.5,0,3*2.598)
--++(3*3,0,0)
               (0,0,-2*3.464)++(-3*1.5,0,3*2.598)--++(3*1.5,3*2.449,-3*0.866);

\draw [line width=0.25pt]
 (0,2.449,0)--++(-1.5,0,2.598)
 (0,2.449,0)--++(1.5,0,2.598)
 (0,2.449,0)--++(3,0,0)

 (0,2.449,0)--++(1.5,0,-2.598)
 (0,2.449,0)--++(-1.5,0,-2.598)
 (0,2.449,0)--++(-3,0,0)

 (0,2.449,0)--++(-1.5,-2.449,0.866)
 (0,2.449,0)--++(1.5,-2.449,0.866)
 (0,2.449,0)--++(0,-2.449,-1.732)

 (0,2.449,0)--++(1.5,2.449,-0.866)
 (0,2.449,0)--++(-1.5,2.449,-0.866)
 (0,2.449,0)--++(0,2.449,1.732)
;

\draw [line width=1pt,white] (-3*1.5,2.449,3*0.866)--++(3*3,0,0)--++(-3*1.5,0,-3*2.598)
               (-3,0,3.464)--++(3*1.5,3*2.449,-3*0.866)--++(0,-3*2.449,-3*1.732);

\draw [line width=1pt,white] (-1.5,3*2.449,0.866)--++(3*1.5,-3*2.449,3*0.866)
              (0,3*2.449,-1.732)--++(3*1.5,-3*2.449,3*0.866);

\draw [line width=1pt,white] (0,4*2.449,0)--++(-4*1.5,-4*2.449,4*0.866)--++(4*3,0,0)--cycle;
\draw [line width=1pt,white] (0,4*2.449,0)--++(4*1.5,-4*2.449,4*0.866);

\draw [line width=1pt,white] (0,4*2.449,0)--(0,0,-2*3.464)--++(4*1.5,0,4*2.598);

\draw [line join=round,line cap=round,line width=0.4pt] (-3*1.5,2.449,3*0.866)--++(3*3,0,0)
               (-3,0,3.464)--++(3*1.5,3*2.449,-3*0.866);

\draw [line join=round,line cap=round,line width=0.3pt] (-3*1.5,2.449,3*0.866)++(3*3,0,0)--++(-3*1.5,0,-3*2.598)
               (-3,0,3.464)++(3*1.5,3*2.449,-3*0.866)--++(0,-3*2.449,-3*1.732);

\draw [line join=round,line cap=round,line width=0.4pt] (-1.5,3*2.449,0.866)--++(3*1.5,-3*2.449,3*0.866);

\draw [line join=round,line cap=round,line width=0.3pt] (0,3*2.449,-1.732)--++(3*1.5,-3*2.449,3*0.866);

\draw [line join=round,line cap=round,line width=0.4pt] (0,4*2.449,0)--++(-4*1.5,-4*2.449,4*0.866)--++(4*3,0,0)--cycle;

\draw [line join=round,line cap=round,line width=0.3pt] (0,4*2.449,0)--(0,0,-2*3.464)--++(4*1.5,0,4*2.598);

\pgfpathcircle{\pgfpointxyz{0}{2.449}{0}}{6pt}
\pgfusepath{fill}

\end{scope}

\end{tikzpicture}
\par\end{center}\caption{\label{fig:Splits}A quotient matroid subdivision of rank $4$.}
\end{figure}

\begin{figure}[th]
\noindent \centering{}\noindent \begin{center}
\def\sizea{0.35}

\begin{tikzpicture}[blowup line/.style={line width=2pt,grey},font=\footnotesize]

\matrix[column sep=0.8cm, row sep=0.6cm]{

\begin{scope}[line cap=round,xshift=0cm,yshift=-0.2cm,rotate=0,scale=\sizea]

\pgfsetxvec{\pgfpointxyz{1.5}{0}{0.5}}
\pgfsetzvec{\pgfpointxyz{-0.2}{0}{1.5}}
\pgfsetyvec{\pgfpointxyz{0}{1.25}{0}}

\draw [line join=round,line cap=round,line width=0.2pt]
 (0,2.449,0)++(-3,0,0)--++(-1.5,-2.449,0.866)--++(3,0,0)
 (0,2.449,0)++(-1.5,0,2.598)++(-1.5,-2.449,0.866)++(-3,0,0)--++(1.5,0,-2.598);

\draw [line width=3pt,white]
 (0,2.449,0)++(-3,0,0)++(-1.5,0,2.598)--++(3,0,0)--++(-1.5,-2.449,0.866);

\draw [line join=round,line cap=round,line width=0.3pt]
 (0,2.449,0)--++(-3,0,0)--++(-1.5,0,2.598)--++(3,0,0)--++(-1.5,-2.449,0.866)--++(1.5,0,-2.598)--cycle
 (0,2.449,0)--++(-1.5,0,2.598)++(-1.5,-2.449,0.866)--++(-3,0,0)++(1.5,0,-2.598)++(-1.5,0,2.598)--++(1.5,2.449,-0.866);

\pgfpathcircle{\pgfpointxyz{0}{2.449}{0}}{6pt}
\pgfusepath{fill}

\path (-4.2,4) node[]{${\displaystyle{
\bigcap_{j\in\left[4\right]-\left\{1\right\} }\left\{ x(S_{j})\ge1\right\}}}$};

\end{scope}

&

\begin{scope}[line cap=round,xshift=0cm,yshift=-0.8cm,rotate=0,scale=\sizea]

\pgfsetxvec{\pgfpointxyz{1.5}{0}{0.5}}
\pgfsetzvec{\pgfpointxyz{-0.2}{0}{1.5}}
\pgfsetyvec{\pgfpointxyz{0}{1.25}{0}}

\draw [line join=round,line cap=round,line width=0.2pt]
 (0,2.449,0)--++(-1.5,0,2.598);
\draw [line width=3pt,white]
 (0,2.449,0)++(1.5,0,2.598)--++(-1.5,2.449,-0.866);

\draw [line join=round,line cap=round,line width=0.3pt]
 (0,2.449,0)--++(1.5,0,2.598)
 (0,2.449,0)--++(0,2.449,1.732);

\draw [line join=round,line cap=round,line width=0.3pt]
 (0,2.449,0)++(0,2.449,1.732)--++(-1.5,-2.449,0.866)--++(3,0,0)--cycle;

\pgfpathcircle{\pgfpointxyz{0}{2.449}{0}}{6pt}
\pgfusepath{fill}

\path (-1.6,5.8) node[]{${\displaystyle{
\bigcap_{j\in\left[4\right]-\left\{2\right\} }\left\{ x(S_{j})\le1\right\}}}$};

\end{scope}

&

\begin{scope}[line cap=round,rotate=0,scale=\sizea,xshift=0cm,yshift=0cm]

\pgfsetxvec{\pgfpointxyz{1.5}{0}{0.5}}
\pgfsetzvec{\pgfpointxyz{-0.2}{0}{1.5}}
\pgfsetyvec{\pgfpointxyz{0}{1.25}{0}}

\draw [line join=round,line cap=round,line width=0.2pt]
 (0,2.449,0)--++(-1.5,-2.449,0.866)--++(3,0,0)
 (0,2.449,0)++(-1.5,-2.449,0.866)--++(-1.5,0,2.598);

\draw [line width=3pt,white]
 (0,2.449,0)++(-1.5,0,2.598)--++(3,0,0)--++(1.5,-2.449,0.866);

\draw [line join=round,line cap=round,line width=0.3pt]
 (0,2.449,0)--++(1.5,-2.449,0.866);

\draw [line join=round,line cap=round,line width=0.3pt]
 (0,2.449,0)--++(-1.5,0,2.598)--++(3,0,0)--cycle
 (0,2.449,0)++(-1.5,0,2.598)--++(-1.5,-2.449,0.866)--++(6,0,0)--++(-1.5,2.449,-0.866)++(1.5,-2.449,0.866)--++(-1.5,0,-2.598);

\pgfpathcircle{\pgfpointxyz{0}{2.449}{0}}{6pt}
\pgfusepath{fill}

\path (-2.6,4) node[]{${\displaystyle{\mathbf{P}:=\begin{Bmatrix}x(S_{1})\le1,\:x(S_{2})\ge1\\x(S_{4})\le1,\:x(S_{3})\ge1\end{Bmatrix}}}$};

\end{scope}

\\

\begin{scope}[line cap=round,xshift=-2cm,scale=1,decoration={markings,mark=at position 0.65 with {\arrow{Stealth}}}]

 \path (-1,1) coordinate(b) node[above]{$2$};
 \path (0,1) coordinate(c) node[above]{$3$};
 \path (1,1) coordinate(d) node[above]{$4$};
 \path (0,0) coordinate(a) node[below]{$1$};

 \foreach \x in {b,c,d}{
          \draw [thick,postaction={decorate}] (a)--(\x);}

 \foreach \x in {a,b,c,d}{
          \fill [black](\x) circle (2.2pt);}
\path (0,-.7) node[]{$G_{1}$};
\end{scope}

&

\begin{scope}[line cap=round,xshift=-0.7cm,scale=1,decoration={markings,mark=at position 0.6 with {\arrow{Stealth}}}]

 \path (-1,0) coordinate(b) node[below]{$1$};
 \path (0,0) coordinate(c) node[below]{$3$};
 \path (1,0) coordinate(d) node[below]{$4$};
 \path (0,1) coordinate(a) node[above]{$2$};

 \foreach \x in {b,c,d}{
          \draw [thick,postaction={decorate}] (\x)--(a);}

 \foreach \x in {a,b,c,d}{
          \fill [black](\x) circle (2.2pt);}
\path (0,-.7) node[]{$G_{5}$};
\end{scope}

&

\begin{scope}[line cap=round,xshift=-0.9cm,scale=1,decoration={markings,mark=at position 0.7 with {\arrow{Stealth}}}]

 \path (-.5,1) coordinate(b) node[above]{$2$};
 \path (0.5,1) coordinate(c) node[above]{$3$};
 \path (-.5,0) coordinate(d) node[below]{$1$};
 \path (0.5,0) coordinate(a) node[below]{$4$};

 \foreach \x in {b,c}{
          \draw [thick,postaction={decorate}] (a)--(\x);}

 \draw [thick,postaction={decorate}] (d)--(b);
 \draw [thick,postaction={decorate}] (d)--(c);

 \foreach \x in {a,b,c,d}{
          \fill [black](\x) circle (2.2pt);}
\path (0,-.7) node[]{$G_{2}\cup G_{3}$};
\end{scope}

\\
};
\end{tikzpicture}
\par\end{center}\caption{\label{fig:Blocks0}Three base polytopes of the matroid subdivision
of Example \ref{exa:4-subdivision}, Figure \ref{fig:Splits}, with
the corresponding directed graphs.}
\end{figure}
\end{example}

Observe that in Example \ref{exa:4-subdivision}, by cutting $\Delta_{S}^{4}$
with $4$ hyperplanes $\left\{ x\left(S-S_{i}\right)=3\right\} $
instead, we obtain the same matroid subdivision because:
\[
\left\{ x\left(S-S_{i}\right)=3\right\} =\left\{ x\left(S_{i}\right)=1\right\} .
\]
 Observe also that the $4$ polytopes 
\[
\bigcap_{j\in\left[4\right]-\left\{ i\right\} }\left\{ x\left(S_{j}\right)\ge1\right\} 
\]
 whose quotients are parallelepipeds positioned at the $4$ corners
of the tetrahedron of Figure \ref{fig:Splits}, cannot be further
cut into full-dimensional base polytopes, and neither can the following
$4$ polytopes whose quotients are smaller tetrahedra positioned at
the centers of the 4 facets of the tetrahedron
\[
\bigcap_{j\in\left[4\right]-\left\{ i\right\} }\left\{ x\left(S_{j}\right)\le1\right\} .
\]

Now, we want to obtain a matroid subdivision of $\Delta_{S}^{4}$
with $\mathrm{Q}=\mathrm{BP}_{\bigoplus_{i\in\left[4\right]}U_{S_{i}}^{1}}$
being a common cell that is finer than the matroid subdivision of
Example \ref{exa:4-subdivision}. Then, we know that we have to cut
each of its $6$ polytopes whose quotients are positioned in the middle
of the $6$ edges of the tetrahedron, with hyperplanes of the form
\[
\left\{ x\left(A\right)=2\right\} \quad\text{for a subset }A\subset S
\]
 because we have already cut $\Delta_{S}^{4}$ with all hyperplanes
of the form $\left\{ x\left(A\right)=1\right\} $ or $\left\{ x\left(A\right)=3\right\} $
and all 6 of those polytopes contain $\mathrm{Q}$, cf. Lemma~\ref{lem:GMHA}(\ref{enu:inequalities}).
Further, each $A$ is a rank-$2$ flat of the matroid $\bigoplus_{i\in\left[4\right]}U_{S_{i}}^{1}$,
and is a union of precisely two of $S_{1},S_{2},S_{3}$, and $S_{4}$.
Thus the hyperplanes we cut with have the following form 
\begin{equation}
\left\{ x\left(S_{J}\right)=2\right\} \quad\text{for a size-}2\text{ subset }J\subset\left[4\right].\label{eq:cutting-w-2}
\end{equation}
 Moreover, by the following lemma, the number of such cutting hyperplanes
cannot exceed $1$, and hence is $1$.
\begin{lem}
\label{lem:splits-D(4,n)}Let $M$ be a rank-$4$ connected matroid
with a rank-$2$ non-degenerate flat $F$. If $L$ is a non-degenerate
flat such that $\mathrm{BP}_{M\left(F\right)}\cap\mathrm{BP}_{M\left(L\right)}$
is a codimension-$2$ face of $\mathrm{BP}_{M}$ that is not contained
in a coordinate hyperplane, then $r\left(L\right)\neq2$.
\end{lem}

\begin{proof}
Since $\mathrm{BP}_{M\left(F\right)}\cap\mathrm{BP}_{M\left(L\right)}=\mathrm{BP}_{M\left(F\right)\cap M\left(L\right)}$
is nonempty, $\left\{ F,L\right\} $ is a \emph{modular pair} by Lemma~\ref{lem:GMHA}(\ref{enu:mod-pair}).
Also, $M\left(F\right)\cap M\left(L\right)$ is a loopless face matroid.
To prove by contrapositive, suppose $r\left(L\right)=2$. Then, $r\left(L\cap F\right)<2$
since $L\neq F$. Since $\left\{ F,L\right\} $ is a modular pair,
we have either $r\left(F\cap L\right)=0$ and $r\left(F\cup L\right)=4$,
or $r\left(F\cap L\right)=1$ and $r\left(F\cup L\right)=3$. In the
former case, $F\cap L=\emptyset$ which implies that $F\cup L$ is
a non-flat of rank $4$ by Lemma~\ref{lem:GMHA}(\ref{enu:codim-2}),
but then $M\left(F\right)\cap M\left(L\right)=M\left(F\cup L\right)$
has a loop, a contradiction. In the latter case, $F\cap L\neq\emptyset$
and $F\cup L\neq S$. Again by Lemma~\ref{lem:GMHA}(\ref{enu:codim-2}),
$F\subsetneq L$ or $L\subsetneq F$, which contradicts $r\left(F\right)=r\left(L\right)=2$.
Thus we conclude $r\left(L\right)\neq2$.
\end{proof}
By Lemmas \ref{lem:Cutting-1} and \ref{lem:splits-D(4,n)}, for any
matroid subdivision of our interest, the number of its members does
not exceed $4+4+\mathbf{2}\cdot6=20$ which is the maximum number
of vertices of a $3$-dimensional biconvex polytope. Moreover, if
we cut each of those $6$ polytopes with a hyperplane of the form
(\ref{eq:cutting-w-2}), then we obtain a matroid subdivision with
$20$ members, see Figures \ref{fig:Blocks1} and \ref{fig:Blocks2};
hence the maximum is attained.
\begin{figure}[th]
\noindent \begin{centering}
\noindent \begin{center}
\def\sizea{0.35}
\def\size{0.25}
\def\ratioa{0.901}
\def\ratior{0.851}
\def\ratiow{0.761}
\def\ratioe{0.881}
\colorlet{grey}{black!37}

\begin{tikzpicture}[blowup line/.style={line width=2pt,grey},font=\scriptsize]

\matrix[column sep=0.4cm, row sep=0.8cm]{

&

\begin{scope}[line cap=round,rotate=0,scale=\sizea,xshift=0cm,yshift=0cm]

\pgfsetxvec{\pgfpointxyz{1.5}{0}{0.5}}
\pgfsetzvec{\pgfpointxyz{-0.2}{0}{1.5}}
\pgfsetyvec{\pgfpointxyz{0}{1.25}{0}}

\draw [line join=round,line cap=round,line width=0.2pt]
 (0,2.449,0)--++(-1.5,-2.449,0.866)--++(3,0,0)
 (0,2.449,0)++(-1.5,-2.449,0.866)--++(-1.5,0,2.598);

\draw [line width=3pt,white]
 (0,2.449,0)++(-1.5,0,2.598)--++(3,0,0)--++(1.5,-2.449,0.866);

\draw [line join=round,line cap=round,line width=0.3pt]
 (0,2.449,0)--++(1.5,-2.449,0.866);

\draw [line join=round,line cap=round,line width=0.3pt]
 (0,2.449,0)--++(-1.5,0,2.598)--++(3,0,0)--cycle
 (0,2.449,0)++(-1.5,0,2.598)--++(-1.5,-2.449,0.866)--++(6,0,0)--++(-1.5,2.449,-0.866)++(1.5,-2.449,0.866)--++(-1.5,0,-2.598);

\draw [line join=round,line cap=round,line width=1.2pt,semitransparent,cyan]
 (0,2.449,0)--++(1.5,0,2.598)--++(-1.5,-2.449,0.866)--++(-1.5,0,-2.598)--++(1.5,2.449,-0.866);

\fill [semitransparent,cyan]
 (0,2.449,0)--++(1.5,0,2.598)--++(-1.5,-2.449,0.866)--++(-1.5,0,-2.598)--++(1.5,2.449,-0.866);

\draw [line join=round,line cap=round,line width=0.3pt]
 (0,2.449,0)--++(1.5,0,2.598)--++(-3,0,0)
 (0,2.449,0)++(1.5,0,2.598)--++(1.5,-2.449,0.866);

\pgfpathcircle{\pgfpointxyz{0}{2.449}{0}}{6pt}
\pgfusepath{fill}

\end{scope}

&

\\\hline
\\

\begin{scope}[line cap=round,rotate=0,xshift=.5cm,yshift=0cm,scale=\sizea]

\pgfsetxvec{\pgfpointxyz{1.5}{0}{0.5}}
\pgfsetzvec{\pgfpointxyz{-0.2}{0}{1.5}}
\pgfsetyvec{\pgfpointxyz{0}{1.25}{0}}

\draw [line join=round,line cap=round,line width=0.3pt]
 (0,2.449,0)--++(1.5,0,2.598)--++(-1.5,-2.449,0.866)--++(-1.5,0,-2.598)--++(1.5,2.449,-0.866);

\draw [line width=3pt,white]
 (0,2.449,0)++(-1.5,0,2.598)++(1,0,0)--++(1.7,0,0);

\draw [line join=round,line cap=round,line width=0.3pt]
 (0,2.449,0)--++(-1.5,0,2.598)--++(3,0,0)
 (0,2.449,0)++(-1.5,0,2.598)--++(-1.5,-2.449,0.866)
 (0,2.449,0)++(-1.5,-2.449,0.866)--++(-1.5,0,2.598)--++(3,0,0);

\pgfpathcircle{\pgfpointxyz{0}{2.449}{0}}{6pt}
\pgfusepath{fill}

\path (-2.4,4.2) node[]{${\displaystyle{\begin{Bmatrix}x(S_{4})\le1\\x(S_{3})\ge1\\x(S_{1}\cup S_{3})\le2\end{Bmatrix}}}$};

\end{scope}

&

\begin{scope}[line cap=round,rotate=0,xshift=0cm,yshift=0cm,scale=\sizea]

\pgfsetxvec{\pgfpointxyz{1.5}{0}{0.5}}
\pgfsetzvec{\pgfpointxyz{-0.2}{0}{1.5}}
\pgfsetyvec{\pgfpointxyz{0}{1.25}{0}}

\fill [semitransparent,cyan]
 (0,2.449,0)--++(1.5,0,2.598)--++(-1.5,-2.449,0.866)--++(-1.5,0,-2.598)--++(1.5,2.449,-0.866);

\draw [line join=round,line cap=round,line width=1.2pt,semitransparent,cyan]
 (0,2.449,0)--++(1.5,0,2.598)--++(-1.5,-2.449,0.866)--++(-1.5,0,-2.598)--++(1.5,2.449,-0.866);

\pgfpathcircle{\pgfpointxyz{0}{2.449}{0}}{6pt}
\pgfusepath{fill}

\path (-1.5,4.2) node[]{${\displaystyle{\begin{Bmatrix}x(S_{4})\le1\\x(S_{3})\ge1\\x(S_{1}\cup S_{3})=2\end{Bmatrix}=\begin{Bmatrix}x(S_{1})\le1\\x(S_{2})\ge1\\x(S_{2}\cup S_{4})=2\end{Bmatrix}}}$};

\end{scope}

&

\begin{scope}[line cap=round,rotate=0,xshift=-.2cm,yshift=0cm,scale=\sizea]

\pgfsetxvec{\pgfpointxyz{1.5}{0}{0.5}}
\pgfsetzvec{\pgfpointxyz{-0.2}{0}{1.5}}
\pgfsetyvec{\pgfpointxyz{0}{1.25}{0}}

\draw [line join=round,line cap=round,line width=0.3pt]
 (0,2.449,0)--++(1.5,0,2.598)++(-1.5,-2.449,0.866)--++(-1.5,0,-2.598)--++(1.5,2.449,-0.866)
 (0,2.449,0)--++(1.5,-2.449,.866)--++(-3,0,0)

(0,2.449,0)++(1.5,-2.449,.866)--++(1.5,0,2.598);

\draw [line width=3pt,white]
 (0,2.449,0)++(1.5,0,2.598)++(.5,-0.816,.289)--++(.5,-0.816,.289)
 (0,2.449,0)++(1.5,0,2.598)++(-.5,-.816,0.289)--++(-.5,-.816,0.289);

\draw [line join=round,line cap=round,line width=0.3pt]
 (0,2.449,0)++(1.5,0,2.598)--++(1.5,-2.449,.866)--++(-3,0,0)
 (0,2.449,0)++(1.5,0,2.598)--++(-1.5,-2.449,0.866);

\pgfpathcircle{\pgfpointxyz{0}{2.449}{0}}{6pt}
\pgfusepath{fill}

\path (-1,4.2) node[]{${\displaystyle{\begin{Bmatrix}x(S_{1})\le1\\x(S_{2})\ge1\\x(S_{2}\cup S_{4})\le2\end{Bmatrix}}}$};

\end{scope}

\\

\begin{scope}[line cap=round,xshift=-.5cm,scale=1,decoration={markings,mark=at position 0.6 with {\arrow{Stealth}}}]

 \path (-.5,1) coordinate(b) node[above]{$3$};
 \path (0.5,1) coordinate(c) node[above]{$2$};
 \path (-.5,0) coordinate(d) node[below]{$1$};
 \path (0.5,0) coordinate(a) node[below]{$4$};

 \foreach \x in {b,c}{
          \draw [thick,postaction={decorate}] (d)--(\x);}

 \draw [thick,postaction={decorate}] (a)--(c);

 \foreach \x in {a,b,c,d}{
          \fill [black](\x) circle (2.2pt);}
\path (0,-.7) node[]{$G_{2}$};

\end{scope}

&

\begin{scope}[line cap=round,xshift=-.5cm,scale=1,decoration={markings,mark=at position 0.65 with {\arrow{Stealth}}}]

 \path (-.5,1) coordinate(b) node[above=0pt]{$3$};
 \path (0.5,1) coordinate(c) node[above=0pt]{$2$};
 \path (-.5,0) coordinate(d) node[below=0pt]{$1$};
 \path (0.5,0) coordinate(a) node[below=0pt]{$4$};

 \draw [thick,postaction={decorate}] (a)--(c);
 \draw [thick,postaction={decorate}] (d)--(b);

 \foreach \x in {a,b,c,d}{
          \fill [black](\x) circle (2.2pt);}
\path (0,-.7) node[]{$G_{2}(1,2)=G_{3}(4,3)$};

\end{scope}

&

\begin{scope}[line cap=round,xshift=-.5cm,scale=1,decoration={markings,mark=at position 0.6 with {\arrow{Stealth}}}]

 \path (-.5,1) coordinate(b) node[above]{$3$};
 \path (0.5,1) coordinate(c) node[above]{$2$};
 \path (-.5,0) coordinate(d) node[below]{$1$};
 \path (0.5,0) coordinate(a) node[below]{$4$};

 \foreach \x in {b,c}{
          \draw [thick,postaction={decorate}] (a)--(\x);}

 \draw [thick,postaction={decorate}] (d)--(b);

 \foreach \x in {a,b,c,d}{
          \fill [black](\x) circle (2.2pt);}
\path (0,-.7) node[]{$G_{3}$};

\end{scope}

\\\hline
\\

\begin{scope}[line cap=round,rotate=0,xshift=0cm,yshift=0cm,scale=\sizea]
\pgfsetxvec{\pgfpointxyz{1.5}{0}{0.5}}
\pgfsetzvec{\pgfpointxyz{-0.2}{0}{1.5}}
\pgfsetyvec{\pgfpointxyz{0}{1.25}{0}}

\draw [line join=round,line cap=round,line width=0.1pt,gray,densely dashed]
 (0,2.449,0)--++(1.5,0,2.598)--++(-1.5,-2.449,0.866)--++(-1.5,0,-2.598);

\draw [line join=round,line cap=round,line width=1.2pt,cyan]
 (0,2.449,0)--++(-1.5,-2.449,.866);

\pgfpathcircle{\pgfpointxyz{0}{2.449}{0}}{6pt}
\pgfusepath{fill}

\path (-1.2,4.2) node[]{${\displaystyle{\begin{Bmatrix}x(S_{1})\le1\\x(S_{3})\ge1\\x(S_{2})=1\\x(S_{4})=1\end{Bmatrix}}}$};

\end{scope}

&

\begin{scope}[line cap=round,rotate=0,xshift=0cm,yshift=0cm,scale=\sizea]
\pgfsetxvec{\pgfpointxyz{1.5}{0}{0.5}}
\pgfsetzvec{\pgfpointxyz{-0.2}{0}{1.5}}
\pgfsetyvec{\pgfpointxyz{0}{1.25}{0}}

\draw [line join=round,line cap=round,line width=0.1pt,gray,densely dashed]
 (0,2.449,0)--++(-1.5,-2.449,0.866)--++(1.5,0,2.598)--++(1.5,2.449,-.866);

\draw [line join=round,line cap=round,line width=1.2pt,cyan]
 (0,2.449,0)--++(1.5,0,2.598);

\pgfpathcircle{\pgfpointxyz{0}{2.449}{0}}{6pt}
\pgfusepath{fill}

\path (-1.2,4.2) node[]{${\displaystyle{\begin{Bmatrix}x(S_{4})\le1\\x(S_{2})\ge1\\x(S_{1})=1\\x(S_{3})=1\end{Bmatrix}}}$};

\end{scope}

&

\begin{scope}[line cap=round,rotate=0,xshift=0cm,yshift=0cm,scale=\sizea]
\pgfsetxvec{\pgfpointxyz{1.5}{0}{0.5}}
\pgfsetzvec{\pgfpointxyz{-0.2}{0}{1.5}}
\pgfsetyvec{\pgfpointxyz{0}{1.25}{0}}

\draw [line join=round,line cap=round,line width=0.1pt,gray,densely dashed]
 (0,2.449,0)--++(1.5,0,2.598)--++(-1.5,-2.449,0.866)--++(-1.5,0,-2.598)--++(1.5,2.449,-0.866);

\pgfpathcircle{\pgfpointxyz{0}{2.449}{0}}{6pt}
\pgfusepath{fill}

\path (-1.2,4.2) node[]{${\displaystyle{\begin{Bmatrix}x(S_{1})=1\\x(S_{2})=1
\\x(S_{3})=1\\x(S_{4})=1\end{Bmatrix}}}$};

\end{scope}

\\

\begin{scope}[line cap=round,xshift=-.5cm,scale=1,decoration={markings,mark=at position 0.65 with {\arrow{Stealth}}}]

 \path (-.5,1) coordinate(b) node[above=0pt]{$3$};
 \path (0.5,1) coordinate(c) node[above=0pt]{$2$};
 \path (-.5,0) coordinate(d) node[below=0pt]{$1$};
 \path (0.5,0) coordinate(a) node[below=0pt]{$4$};

 \draw [thick,postaction={decorate}] (d)--(b);

 \foreach \x in {a,b,c,d}{
          \fill [black](\x) circle (2.2pt);}
\path (0,-1) node[]{$\begin{array}{c}G_{2}(1,2)(4,2)\\\parallel\\G_{3}(4,3)(4,2)\end{array}$};

\end{scope}

&

\begin{scope}[line cap=round,xshift=-.5cm,scale=1,decoration={markings,mark=at position 0.65 with {\arrow{Stealth}}}]

 \path (-.5,1) coordinate(b) node[above=0pt]{$3$};
 \path (0.5,1) coordinate(c) node[above=0pt]{$2$};
 \path (-.5,0) coordinate(d) node[below=0pt]{$1$};
 \path (0.5,0) coordinate(a) node[below=0pt]{$4$};

 \draw [thick,postaction={decorate}] (a)--(c);

 \foreach \x in {a,b,c,d}{
          \fill [black](\x) circle (2.2pt);}
\path (0,-1) node[]{$\begin{array}{c}G_{2}(1,2)(1,3)\\\parallel\\G_{3}(4,3)(1,3)\end{array}$};

\end{scope}

&

\begin{scope}[line cap=round,xshift=-0.5cm,scale=1]
 \path (-.5,1) coordinate(b) node[above=0pt]{$3$};
 \path (0.5,1) coordinate(c) node[above=0pt]{$2$};
 \path (-.5,0) coordinate(d) node[below=0pt]{$1$};
 \path (0.5,0) coordinate(a) node[below=0pt]{$4$};

 \foreach \x in {a,b,c,d}{
          \fill [black](\x) circle (2.2pt);}
\path (0,-1) node[]{$\begin{array}{c}G_{2}(1,2)(1,3)(4,2)\\\parallel\\G_{3}(4,3)(1,3)(4,2)\end{array}$};

\end{scope}

\\
};
\end{tikzpicture}
\par\end{center}
\par\end{centering}
\noindent \centering{}\caption{\label{fig:Blocks1}A matroid subdivision of $\mathrm{P}$ with the
corresponding directed bigraphs, I.}
\end{figure}
\begin{figure}[th]
\noindent \begin{centering}
\noindent \begin{center}
\def\sizea{0.35}
\def\size{0.25}
\def\ratioa{0.901}
\def\ratior{0.851}
\def\ratiow{0.761}
\def\ratioe{0.881}
\colorlet{grey}{black!37}

\begin{tikzpicture}[blowup line/.style={line width=2pt,grey},font=\scriptsize]

\matrix[column sep=0.4cm, row sep=0.8cm]{

&

\begin{scope}[line cap=round,rotate=0,scale=\sizea,xshift=0cm,yshift=0cm]

\pgfsetxvec{\pgfpointxyz{1.5}{0}{0.5}}
\pgfsetzvec{\pgfpointxyz{-0.2}{0}{1.5}}
\pgfsetyvec{\pgfpointxyz{0}{1.25}{0}}

\draw [line join=round,line cap=round,line width=0.2pt]
 (0,2.449,0)--++(-1.5,-2.449,0.866)--++(3,0,0)
 (0,2.449,0)++(-1.5,-2.449,0.866)--++(-1.5,0,2.598);

\draw [line width=3pt,white]
 (0,2.449,0)++(-1.5,0,2.598)--++(3,0,0)--++(1.5,-2.449,0.866);

\draw [line join=round,line cap=round,line width=0.3pt]
 (0,2.449,0)--++(1.5,-2.449,0.866);

\draw [line join=round,line cap=round,line width=0.3pt]
 (0,2.449,0)--++(-1.5,0,2.598)--++(3,0,0)--cycle
 (0,2.449,0)++(-1.5,0,2.598)--++(-1.5,-2.449,0.866)--++(6,0,0)--++(-1.5,2.449,-0.866)++(1.5,-2.449,0.866)--++(-1.5,0,-2.598);

\fill [semitransparent,cyan]
 (0,2.449,0)--++(-1.5,0,2.598)--++(1.5,-2.449,0.866)--++(1.5,0,-2.598)--++(-1.5,2.449,-0.866);

\draw [line join=round,line cap=round,line width=1.2pt,semitransparent,cyan]
 (0,2.449,0)--++(-1.5,0,2.598)--++(1.5,-2.449,0.866)--++(1.5,0,-2.598)--++(-1.5,2.449,-0.866);

\draw [line join=round,line cap=round,line width=0.3pt]
 (0,2.449,0)--++(1.5,0,2.598)--++(-3,0,0)
 (0,2.449,0)++(1.5,0,2.598)--++(1.5,-2.449,0.866);

\pgfpathcircle{\pgfpointxyz{0}{2.449}{0}}{6pt}
\pgfusepath{fill}

\end{scope}

&

\\\hline
\\

\begin{scope}[line cap=round,rotate=0,xshift=0cm,yshift=0cm,scale=\sizea]

\pgfsetxvec{\pgfpointxyz{1.5}{0}{0.5}}
\pgfsetzvec{\pgfpointxyz{-0.2}{0}{1.5}}
\pgfsetyvec{\pgfpointxyz{0}{1.25}{0}}

\draw [line join=round,line cap=round,line width=0.2pt]
 (0,2.449,0)--++(-1.5,-2.449,0.866)--++(3,0,0)
 (0,2.449,0)++(-1.5,-2.449,0.866)--++(-1.5,0,2.598)--++(3,0,0)
 (0,2.449,0)++(-1.5,-2.449,0.866)++(-1.5,0,2.598)--++(1.5,2.449,-0.866);

\draw [line width=3pt,white]
 (0,2.449,0)++(-1.5,0,2.598)
++(0.5,-.816,0.289)--
++(0.5,-.816,0.289);

\draw [line join=round,line cap=round,line width=0.2pt]
 (0,2.449,0)--++(-1.5,0,2.598)--++(1.5,-2.449,0.866)--++(1.5,0,-2.598)--++(-1.5,2.449,-0.866);

\pgfpathcircle{\pgfpointxyz{0}{2.449}{0}}{6pt}
\pgfusepath{fill}

\path (-1.8,4.2) node[]{${\displaystyle{\begin{Bmatrix}x(S_{4})\le1\\x(S_{2})\ge1\\x(S_{1}\cup S_{2})\le2\end{Bmatrix}}}$};

\end{scope}

&

\begin{scope}[line cap=round,rotate=0,xshift=0cm,yshift=0cm,scale=\sizea]
\pgfsetxvec{\pgfpointxyz{1.5}{0}{0.5}}
\pgfsetzvec{\pgfpointxyz{-0.2}{0}{1.5}}
\pgfsetyvec{\pgfpointxyz{0}{1.25}{0}}

\fill [semitransparent,cyan]
 (0,2.449,0)--++(-1.5,0,2.598)--++(1.5,-2.449,0.866)--++(1.5,0,-2.598)--++(-1.5,2.449,-0.866);

\draw [line join=round,line cap=round,line width=1.2pt,semitransparent,cyan]
 (0,2.449,0)--++(-1.5,0,2.598)--++(1.5,-2.449,0.866)--++(1.5,0,-2.598)--++(-1.5,2.449,-0.866);

\pgfpathcircle{\pgfpointxyz{0}{2.449}{0}}{6pt}
\pgfusepath{fill}

\path (-1.2,4.2) node[]{${\displaystyle{\begin{Bmatrix}x(S_{4})\le1\\x(S_{2})\ge1\\x(S_{1}\cup S_{2})=2\end{Bmatrix}=\begin{Bmatrix}x(S_{1})\le1\\x(S_{3})\ge1\\x(S_{3}\cup S_{4})=2\end{Bmatrix}}}$};

\end{scope}

&

\begin{scope}[line cap=round,rotate=0,xshift=0cm,yshift=0cm,scale=\sizea]

\pgfsetxvec{\pgfpointxyz{1.5}{0}{0.5}}
\pgfsetzvec{\pgfpointxyz{-0.2}{0}{1.5}}
\pgfsetyvec{\pgfpointxyz{0}{1.25}{0}}

\draw [line join=round,line cap=round,line width=0.2pt]
 (0,2.449,0)--++(-1.5,0,2.598)--++(1.5,-2.449,0.866)--++(1.5,0,-2.598)--++(-1.5,2.449,-0.866);

\draw [line width=3pt,white]
 (0,2.449,0)++(-1.5,0,2.598)++(3,0,0)++(.5,-.816,0.289)--++(.5,-.816,0.289);

\draw [line join=round,line cap=round,line width=0.3pt]
 (0,2.449,0)++(-1.5,0,2.598)--++(3,0,0)--++(-1.5,0,-2.598)

 (0,2.449,0)++(-1.5,0,2.598)++(1.5,-2.449,0.866)--++(3,0,0)--++(-1.5,0,-2.598)

 (0,2.449,0)++(-1.5,0,2.598)++(3,0,0)--++(1.5,-2.449,0.866);

\pgfpathcircle{\pgfpointxyz{0}{2.449}{0}}{6pt}
\pgfusepath{fill}

\path (-.8,4.2) node[]{${\displaystyle{\begin{Bmatrix}x(S_{1})\le1\\x(S_{3})\ge1\\x(S_{3}\cup S_{4})\le2\end{Bmatrix}}}$};

\end{scope}

\\

\begin{scope}[line cap=round,xshift=-.75cm,scale=1,decoration={markings,mark=at position 0.6 with {\arrow{Stealth}}}]

\path (-.5,1) coordinate(b) node[above]{$2$};
 \path (0.5,1) coordinate(c) node[above]{$3$};
 \path (-.5,0) coordinate(d) node[below]{$1$};
 \path (0.5,0) coordinate(a) node[below]{$4$};

 \foreach \x in {b,c}{
          \draw [thick,postaction={decorate}] (d)--(\x);}

 \draw [thick,postaction={decorate}] (a)--(c);

 \foreach \x in {a,b,c,d}{
          \fill [black](\x) circle (2.2pt);}
\path (0,-.7) node[]{$G'_{2}$};
\end{scope}

&

\begin{scope}[line cap=round,xshift=-.5cm,scale=1,decoration={markings,mark=at position 0.65 with {\arrow{Stealth}}}]

 \path (-.5,1) coordinate(b) node[above=0pt]{$2$};
 \path (0.5,1) coordinate(c) node[above=0pt]{$3$};
 \path (-.5,0) coordinate(d) node[below=0pt]{$1$};
 \path (0.5,0) coordinate(a) node[below=0pt]{$4$};

 \draw [thick,postaction={decorate}] (a)--(c);
 \draw [thick,postaction={decorate}] (d)--(b);

 \foreach \x in {a,b,c,d}{
          \fill [black](\x) circle (2.2pt);}
\path (0,-.7) node[]{$G_{2}'(1,3)=G'_{3}(4,2)$};

\end{scope}

&

\begin{scope}[line cap=round,xshift=-.4cm,scale=1,decoration={markings,mark=at position 0.6 with {\arrow{Stealth}}}]

 \path (-.5,1) coordinate(b) node[above]{$2$};
 \path (0.5,1) coordinate(c) node[above]{$3$};
 \path (-.5,0) coordinate(d) node[below]{$1$};
 \path (0.5,0) coordinate(a) node[below]{$4$};

 \foreach \x in {b,c}{
          \draw [thick,postaction={decorate}] (a)--(\x);}

 \draw [thick,postaction={decorate}] (d)--(b);

 \foreach \x in {a,b,c,d}{
          \fill [black](\x) circle (2.2pt);}
\path (0,-.7) node[]{$G'_{3}$};

\end{scope}

\\\hline
\\

\begin{scope}[line cap=round,rotate=0,xshift=0cm,yshift=0cm,scale=\sizea]
\pgfsetxvec{\pgfpointxyz{1.5}{0}{0.5}}
\pgfsetzvec{\pgfpointxyz{-0.2}{0}{1.5}}
\pgfsetyvec{\pgfpointxyz{0}{1.25}{0}}

\draw [line join=round,line cap=round,line width=0.1pt,gray,densely dashed]
 (0,2.449,0)--++(1.5,-2.449,0.866)
--++(-1.5,0,2.598)--++(-1.5,2.449,-.866);

\draw [line join=round,line cap=round,line width=1.2pt,cyan]
 (0,2.449,0)--++(-1.5,0,2.598);

\pgfpathcircle{\pgfpointxyz{0}{2.449}{0}}{6pt}
\pgfusepath{fill}

\path (-1.2,4.3) node[]{${\displaystyle{\begin{Bmatrix}x(S_{1})\le1\\x(S_{2})\ge1\\x(S_{3})=1\\x(S_{4})=1\end{Bmatrix}}}$};

\end{scope}

&

\begin{scope}[line cap=round,rotate=0,xshift=0cm,yshift=0cm,scale=\sizea]
\pgfsetxvec{\pgfpointxyz{1.5}{0}{0.5}}
\pgfsetzvec{\pgfpointxyz{-0.2}{0}{1.5}}
\pgfsetyvec{\pgfpointxyz{0}{1.25}{0}}

\draw [line join=round,line cap=round,line width=0.1pt,gray,densely dashed]
 (0,2.449,0)--++(-1.5,0,2.598)--++(1.5,-2.449,0.866)--++(1.5,0,-2.598);

\draw [line join=round,line cap=round,line width=1.2pt,cyan]
 (0,2.449,0)--++(1.5,-2.449,0.866);

\pgfpathcircle{\pgfpointxyz{0}{2.449}{0}}{6pt}
\pgfusepath{fill}

\path (-1.2,4.2) node[]{${\displaystyle{\begin{Bmatrix}x(S_{4})\le1\\x(S_{3})\ge1\\x(S_{1})=1\\x(S_{2})=1\end{Bmatrix}}}$};

\end{scope}

&

\begin{scope}[line cap=round,rotate=0,xshift=0cm,yshift=0cm,scale=\sizea]
\pgfsetxvec{\pgfpointxyz{1.5}{0}{0.5}}
\pgfsetzvec{\pgfpointxyz{-0.2}{0}{1.5}}
\pgfsetyvec{\pgfpointxyz{0}{1.25}{0}}

\draw [line join=round,line cap=round,line width=0.1pt,gray,densely dashed]
 (0,2.449,0)--++(-1.5,0,2.598)--++(1.5,-2.449,0.866)--++(1.5,0,-2.598)--++(-1.5,2.449,-0.866);

\pgfpathcircle{\pgfpointxyz{0}{2.449}{0}}{6pt}
\pgfusepath{fill}

\path (-1.2,4.3) node[]{${\displaystyle{\begin{Bmatrix}x(S_{1})=1\\x(S_{2})=1
\\x(S_{3})=1\\x(S_{4})=1\end{Bmatrix}}}$};

\end{scope}

\\

\begin{scope}[line cap=round,xshift=-.5cm,scale=1,decoration={markings,mark=at position 0.65 with {\arrow{Stealth}}}]

 \path (-.5,1) coordinate(b) node[above=0pt]{$2$};
 \path (0.5,1) coordinate(c) node[above=0pt]{$3$};
 \path (-.5,0) coordinate(d) node[below=0pt]{$1$};
 \path (0.5,0) coordinate(a) node[below=0pt]{$4$};

 \draw [thick,postaction={decorate}] (d)--(b);

 \foreach \x in {a,b,c,d}{
          \fill [black](\x) circle (2.2pt);}
\path (0,-1) node[]{$\begin{array}{c}G'_{2}(1,3)(4,3)\\\parallel\\G'_{3}(4,2)(4,3)\end{array}$};

\end{scope}

&

\begin{scope}[line cap=round,xshift=-.5cm,scale=1,decoration={markings,mark=at position 0.65 with {\arrow{Stealth}}}]

 \path (-.5,1) coordinate(b) node[above=0pt]{$2$};
 \path (0.5,1) coordinate(c) node[above=0pt]{$3$};
 \path (-.5,0) coordinate(d) node[below=0pt]{$1$};
 \path (0.5,0) coordinate(a) node[below=0pt]{$4$};

 \draw [thick,postaction={decorate}] (a)--(c);

 \foreach \x in {a,b,c,d}{
          \fill [black](\x) circle (2.2pt);}
\path (0,-1) node[]{$\begin{array}{c}G'_{2}(1,3)(1,2)\\\parallel\\G'_{3}(4,2)(1,2)\end{array}$};

\end{scope}

&

\begin{scope}[line cap=round,xshift=-0.5cm,scale=1]
 \path (-.5,1) coordinate(b) node[above=0pt]{$2$};
 \path (0.5,1) coordinate(c) node[above=0pt]{$3$};
 \path (-.5,0) coordinate(d) node[below=0pt]{$1$};
 \path (0.5,0) coordinate(a) node[below=0pt]{$4$};

 \foreach \x in {a,b,c,d}{
          \fill [black](\x) circle (2.2pt);}
\path (0,-1) node[]{$\begin{array}{c}G'_{2}(1,3)(1,2)(4,3)\\\parallel\\G'_{3}(1,2)(4,2)(4,3)\end{array}$};

\end{scope}

\\
};
\end{tikzpicture}
\par\end{center}
\par\end{centering}
\noindent \centering{}\caption{\label{fig:Blocks2}A matroid subdivision of $\mathrm{P}$ with the
corresponding directed bigraphs, II.}
\end{figure}

\appendix

\section{\label{sec:Preliminaries}Definitions and Lemmas}

Throughout the section, $M$ is a matroid on a finite set $S$ with
rank function $r$.\smallskip{}

A pair $\left\{ A,B\right\} $ of subsets of $S$ is called a \textbf{modular
pair} of $M$ if equality holds in the submodular inequality $r\left(A\right)+r\left(B\right)\ge r\left(A\cup B\right)+r\left(A\cap B\right)$,
that is:
\[
r\left(A\right)+r\left(B\right)=r\left(A\cup B\right)+r\left(A\cap B\right).
\]

A subset $A\subseteq S$ is called a \textbf{separator} of $M$ if
$\left\{ A,S-A\right\} $ is a modular pair. Both $S$ and $\emptyset$
are separators which are called \textbf{trivial} separators. Then,
$M$ and its dual matroid $M^{\ast}$ have the same set of separators.\smallskip{}

The matroid $M$ is called \textbf{connected} if it has no nontrivial
separators, and \textbf{disconnected} otherwise. A subset $A\subseteq S$
is called \textbf{connected} if the restriction matroid $M|_{A}$
is, and \textbf{disconnected} otherwise.\smallskip{}

We denote by $\kappa\left(M\right)$ the number of all nonempty inclusionwise
minimal separators of $M$ where the inclusion is set inclusion. Let
$S_{1},\dots,S_{\kappa\left(M\right)}$ be all nonempty inclusionwise
minimal separators of $M$, then:
\[
M=M|_{S_{1}}\oplus\cdots\oplus M|_{S_{\kappa\left(M\right)}}.
\]
 Here, $M|_{S_{i}}$ are called the \textbf{connected components}
of $M$.\smallskip{}

For a subset $A\subseteq S$, we denote
\begin{equation}
M\left(A\right):=M|_{A}\oplus M/A.\label{eq:paren}
\end{equation}

The subset $A$ is called \textbf{non-degenerate}\footnote{The definition of non-degenerate subsets was originally given in \cite{GS87}
for connected matroids, and generalized to the current form in \cite{j-hope}.} if 
\[
\kappa\left(M\left(A\right)\right)=\kappa\left(M\right)+1
\]
 and \textbf{degenerate} otherwise. Every separator is degenerate.
If $A$ is a non-degenerate subset of $M$, then $S-A$ is a non-degenerate
subset of its dual matroid $M^{\ast}$.\smallskip{}

If $M$ is a disconnected matroid, there can be different non-degenerate
subsets $A_{1},\dots,A_{m}$ with $M\left(A_{1}\right)=\cdots=M\left(A_{m}\right)$,
but there exists the smallest such.\smallskip{}

The \textbf{indicator vector} of $A\subseteq S$ is the vector $\mathbf{v}\in\mathbb{R}^{S}$
such that $x_{i}\left(\mathbf{v}\right)$ is $1$ if $i\in A$ and
$0$ otherwise, which we denote by $1^{A}$.\smallskip{}

The set of the bases of $M$ is denoted by $\mathcal{B}\left(M\right)$.
The \textbf{matroid base polytope} or simply the \textbf{base polytope}
of $M$ is the convex hull of all indicator vectors $1^{B}$ of $B\in\mathcal{B}\left(M\right)$,
which we denote by $\mathrm{BP}_{M}$. Its dimension is $\dim\mathrm{BP}_{M}=\left|S\right|-\kappa\left(M\right)$.
Using inequalities, $\mathrm{BP}_{M}$ is written as 
\[
\mathrm{BP}_{M}=\left\{ x_{i}\ge0:i\in S\right\} \cap\left\{ x\left(A\right)\le r\left(A\right):A\in2^{S}\right\} \cap\left\{ x\left(S\right)=r\left(S\right)\right\} .
\]
 The correspondence between matroids and base polytopes is one-to-one.
The base polytope $\mathrm{BP}_{M}$ is full-dimensional if and only
if $M$ is connected.\smallskip{}

A \textbf{face matroid} of $M$ is the matroid of a face of $\mathrm{BP}_{M}$.\smallskip{}

Two polytopes are called \textbf{face-fitting} if their intersection
is a common face of both, empty or not.\smallskip{}

A $\left(k,S\right)$-\textbf{tiling} or simply a \textbf{tiling}
is a finite face-fitting collection of polytopes in $\Delta_{S}^{k}$
that is connected in codimension $1$. The \textbf{support} $\left|\Sigma\right|$
of a tiling $\Sigma$ is the union of its members. The \textbf{dimension}
of $\Sigma$ is the dimension of $\left|\Sigma\right|$. Throughout
the paper, a tiling is assumed equidimensional, i.e. all of its members
have the same dimension. A tiling induced by a convex or concave function
is called \textbf{regular}.\smallskip{}

When mentioning \emph{cells} of $\Sigma$, we identify $\Sigma$ with
the polyhedral complex that its polytopes generate with intersections.
A nonempty cell of $\Sigma$ is called a \textbf{common cell} if it
is a face of all members of $\Sigma$.\smallskip{}

A \textbf{matroid tiling} is a tiling whose members are base polytopes,
which is well defined because every face of a base polytope is again
a base polytope.\smallskip{}

A \textbf{matroid subdivision} is a matroid tiling whose support is
a base polytope.\smallskip{}

The \textbf{base intersection} of two matroids $M_{1}$ and $M_{2}$
is the intersection of the base collections of $M_{1}$ and $M_{2}$,
which we denote by 
\[
M_{1}\cap M_{2}.
\]
 When $M_{1}\cap M_{2}$ is the base collection of a matroid, by abuse
of notation, we denote the matroid by $M_{1}\cap M_{2}$.\smallskip{}

For a subcollection $\mathcal{A}$ of the power set $2^{S}$ of $S$,
let $\mathrm{P}_{\mathcal{A}}$ be the convex hull of the indicator
vectors $1^{A}$ for all $A\in\mathcal{A}$. Then:
\[
\mathrm{BP}_{M_{1}}\cap\mathrm{BP}_{M_{2}}=\mathrm{P}_{M_{1}\cap M_{2}}.
\]

\begin{lem}
\label{lem:GMHA}Let $M$ be a rank-$k$ matroid on $S$ with rank
function $r$.
\begin{enumerate}[itemsep=3pt]
\item \label{enu:inseparable}\cite[Lemma 4.2]{j-hope} If there is a subset
$A\subseteq S$ of size $k+1$ with $M|_{A}=U_{A}^{k}$, then $M\backslash\overline{\emptyset}_{M}$
is a connected matroid where $\overline{\emptyset}_{M}$ denotes the
set of loops of $M$.
\item \label{enu:inequalities}\cite[Lemma 2.30]{j-hope} Suppose $M=M|_{S_{1}}\oplus\cdots\oplus M|_{S_{\kappa\left(M\right)}}$
is loopless. Then, its base polytope $\mathrm{BP}_{M}$ is determined
by $\kappa\left(M\right)$ equations $x\left(S_{i}\right)=r\left(S_{i}\right)$
and the following inequalities:
\[
\ensuremath{\begin{cases}
x_{i}\ge0 & \mbox{ for all }i\in S,\\
x\left(F\right)\le r\left(F\right) & \mbox{ for all minimal non-degenerate flats }F\mbox{ of }M.
\end{cases}}
\]
\item \label{enu:mod-pair}\cite[Lemma 2.24]{j-hope} The pair $\left\{ F,L\right\} $
of subsets of $S$ is modular if and only if $M\left(F\right)\cap M\left(L\right)\neq\emptyset$.
\item \label{enu:codim-2}\cite[Lemma 2.38]{j-hope} Suppose $M$ is connected
with $k\ge3$. Let $\left\{ F,L\right\} $ be a modular pair of non-degenerate
flats with $\mathrm{codim}\,\mathrm{BP}_{M\left(F\right)\cap M\left(L\right)}=2$.
Then, precisely one of the following $4$ cases happens.\vspace{0.6em}

\noindent %
\begin{tabular}{c|c}
\hline 
\scalebox{0.9}{$F\cap L=\emptyset$} & \scalebox{0.9}{$M\left(F\right)\cap M\left(L\right)=M\left(F\cup L\right)$ \, with \, $M|_{F\cup L}= M|_{F}\oplus M|_{L}$}\tabularnewline[\doublerulesep]
\hline 
\scalebox{0.9}{$F\cup L=S$} & \scalebox{0.9}{$M\left(F\right)\cap M\left(L\right)=M\left(F\cap L\right)$ \, with \, $M/\left(F\cap L\right)=M/F\oplus M/L$}\tabularnewline[\doublerulesep]
\hline 
\scalebox{0.9}{$F\supsetneq L$} & \scalebox{0.9}{$M\left(F\right)\cap M\left(L\right)=M/F\oplus M|_F/L\oplus M|_L$}\tabularnewline[\doublerulesep]
\hline 
\scalebox{0.9}{$F\subsetneq L$} & \scalebox{0.9}{$M\left(F\right)\cap M\left(L\right)=M/L\oplus M|_L/F\oplus M|_F$}\tabularnewline[\doublerulesep]
\hline 
\end{tabular}\vspace{1em}

\end{enumerate}
\end{lem}

\end{document}